\documentclass[12pt]{amsart}
\usepackage[cmtip,all]{xy}
\usepackage{graphicx}
\usepackage{amssymb}
\usepackage{mathrsfs}
\usepackage{amsfonts}
\usepackage{amsmath}

\newtheorem{theorem}{Theorem}[section]
\newtheorem{lemma}[theorem]{Lemma}

\theoremstyle{definition}

 \theoremstyle{remark}
\newtheorem{remark}[theorem]{Remark}

 \numberwithin{equation}{section}

\begin{document}

\title[Green's function and sharp  Hardy-Sobolev-Maz'ya inequalities ]{Green's functions of Paneitz and GJMS
 operators on hyperbolic spaces and sharp  Hardy-Sobolev-Maz'ya inequalities on half spaces}

\author{Guozhen Lu}
\address{Department of Mathematics,  University of Connecticut, Storrs, CT 06269, USA}

\email{guozhen.lu@uconn.edu}

\author{Qiaohua  Yang}
\address{School of Mathematics and Statistics, Wuhan University, Wuhan, 430072, People's Republic of China}

\email{qhyang.math@whu.edu.cn; qhyang.math@gmail.com}

\thanks{ The first author was partly supported by a Simons collaboration grant from the Simons Foundation  and the second author's research was partly supported by   the National Natural
Science Foundation of China (No.11201346).}






\begin{abstract}
Using the Fourier analysis techniques on hyperbolic spaces and Green's function estimates, we confirm in this paper  the conjecture given by the same authors in \cite{LuYang3}. Namely, we prove that the sharp constant in the $\frac{n-1}{2}$-th order Hardy-Sobolev-Maz'ya inequality  in the   upper half space of dimension $n$ coincides with the best $\frac{n-1}{2}$-th order Sobolev constant when $n$ is odd and $n\geq9$ (See Theorem \ref{th1.1}).
 We will also establish a lower bound of  the coefficient of the Hardy term for the $k-$th order  Hardy-Sobolev-Maz'ya inequality in upper half space in the remaining cases of dimension $n$ and $k$-th order derivatives (see Theorem \ref{th1.2}).
 Precise expressions and optimal bounds  for Green's functions of the operator $ -\Delta_{\mathbb{H}}-\frac{(n-1)^{2}}{4}$ on the hyperbolic space $\mathbb{B}^n$  and operators of the product form  are given,
 where $\frac{(n-1)^{2}}{4}$ is the spectral gap for the Laplacian  $-\Delta_{\mathbb{H}}$ on  $\mathbb{B}^n$. Finally,
   we give the precise expression and  optimal pointwise bound  of Green's function of the Paneitz and GJMS operators on hyperbolic space, which are of their independent interest (see Theorem \ref{th1.3}).
\end{abstract}

\maketitle


\section{Introduction}
The classical Sobolev inequalities and sharp constants play an essential role in analysis and geometry.  If $k$ is a positive integer, then the $k-$th order Sobolev inequality states as follows (see \cite{au,ta,lie,ct})
 \begin{equation}\label{1.1}
\int_{\mathbb{R}^{n}}|(-\Delta)^{k/2}u|^{2}dx\geq S_{n,k}
\left(\int_{\mathbb{R}^{n}}|u|^{\frac{2n}{n-2k}}dx\right)^{\frac{n-2k}{n}},\;\; u\in C^{\infty}_{0}(\mathbb{R}^{n}),\;\;1\leq k<\frac{n}{2},
\end{equation}
where
 \begin{equation}\label{1.2}
S_{n,k}=\frac{\Gamma(\frac{n+2k}{2})}{\Gamma(\frac{n-2k}{2})}\omega^{2k/n}_{n}
\end{equation}
is the best Sobolev
constant and  $\omega_{n}=\frac{2\pi^{\frac{n+1}{2}}}{\Gamma(\frac{n+1}{2})}$ is the surface measure of $\mathbb{S}^{n}=\{x\in\mathbb{R}^{n+1}: |x|=1\}$. Maz'ya (\cite{maz}, Section 2.1.6) gave
a refinement of both the first order Sobolev and the Hardy inequalities in half spaces which are known as the Hardy-Sobolev-Maz'ya inequality.  It reads as follows
\begin{equation}\label{1.3}
\int_{\mathbb{R}^{n}_{+}}|\nabla u|^{2}dx-\frac{1}{4}\int_{\mathbb{R}^{n}_{+}}\frac{u^{2}}{x^{2}_{1}}dx\geq C_n\left(\int_{\mathbb{R}^{n}_{+}}
x^{\gamma}_{1}|u|^{p}dx\right)^{\frac{2}{p}},\;\;\; u\in C^{\infty}_{0}(\mathbb{R}^{n}_{+}),
\end{equation}
where $2<p\leq\frac{2n}{n-2}$, $\gamma=\frac{(n-2)p}{2}-n$, $\mathbb{R}^{n}_{+}=\{(x_{1},x_{2},\cdots,x_{n})\in\mathbb{R}^{n}: x_{1}>0\}$ and $C_n$ is a positive constant which is independent of $u$. Such inequalities with a Hardy remaining term have been extensively  studied by many authors (see e.g. \cite{be, be2, be3, bfl, BM97, BV, fmt2,fl, LamLuZhang1, m1,mas,te} and the references therein).

 It is known that the upper half space can be regarded as a hyperbolic space. Then the Hardy-Sobolev-Maz'ya  inequality on the upper half space is equivalent to the corresponding inequality on the hyperbolic ball.
 To be more precise,
 we now recall the ball model as a hyperbolic space.
 The unit ball
\[\mathbb{B}^{n}=\{x=(x_{1},\cdots,x_{n})\in \mathbb{R}^{n}| |x|<1\}\]
equipped with the usual Poincar\'e metric
\[
ds^{2}=\frac{4(dx^{2}_{1}+\cdots+dx^{2}_{n})}{(1-|x|^{2})^{2}}
\]
is known as the hyperbolic space of ball model.
The hyperbolic gradient is $\nabla_{\mathbb{H}}=\frac{1-|x|^{2}}{2}\nabla$ and
the
Laplace-Beltrami operator is given by
$$
\Delta_{\mathbb{H}}=\frac{1-|x|^{2}}{4}\left\{(1-|x|^{2})\sum^{n}_{j=1}\frac{\partial^{2}}{\partial
x^{2}_{j}}+2(n-2)\sum^{n}_{j=1}x_{j}\frac{\partial}{\partial
x_{j}}\right\}.
$$
The hyperbolic volume is $dV=\left(\frac{2}{1-|x|^2}\right)^ndx$.

\medskip

Then the GJMS operators on $\mathbb{B}^{n}$  is defined as follows (see \cite{GJMS}, \cite{go},  \cite{j})
\begin{equation}\label{1.5}
 P_{k}=P_{1}(P_{1}+2)\cdot\cdots\cdot(P_{1}+k(k-1)),\;\; k\in\mathbb{N},
\end{equation}
where $P_{1}=-\Delta_{\mathbb{H}}-\frac{n(n-2)}{4}$ is the conformal Laplacian on $\mathbb{B}^{n}$ and
$$P_{2}=\left(-\Delta_{\mathbb{H}}-\frac{n(n-2)}{4}\right)
\left(-\Delta_{\mathbb{H}}-\frac{(n+2)(n-4)}{4}\right)$$
is the well-known Paneitz operator (see \cite{Paneitz}).
 The GJMS operator $P_{k}$  on $\mathbb{B}^{n}$ can also be written as
\begin{equation*}
P_{k}=\prod^{k}_{j=1}\left(-\Delta_{\mathbb{H}}-\frac{(n-1)^{2}}{4}
+(j-1/2)^{2}\right).
\end{equation*}

Hebey proved in \cite{heb}  (see also \cite{heb1}) the following theorem:
\begin{theorem}\label{Hebey-Constant}
Let $n\geq 4$. Suppose that there exists $\lambda\in\mathbb{R}$ such that for
any $u\in C^{\infty}_{0}(\mathbb{B}^{n})$,
\[
\int_{\mathbb{B}^{n}}(P_{1}u)udV+\lambda \int_{\mathbb{B}^{n}}u^{2}dV\geq S_{n, 1}
\left(\int_{\mathbb{B}^{n}}|u|^{\frac{2n}{n-2}}dV\right)^{\frac{n-2}{n}},
\]
where $S_{n, 1}$ is the best  first order Sobolev constant. Then $\lambda\geq0$.
\end{theorem}

In case $n=3$, Benguria, Frank and  Loss (\cite{bfl}) proved the following
\begin{theorem}\label{BFL-Constant}
There holds, for
any $u\in C^{\infty}_{0}(\mathbb{B}^{n})$,
\[
\int_{\mathbb{B}^{3}}(P_{1}u)udV-\frac{1}{4}\int_{\mathbb{B}^{3}}u^{2}dV\geq S_{3, 1}
\left(\int_{\mathbb{B}^{n}}|u|^{6}dV\right)^{\frac{1}{3}},
\]
where $S_{3, 1}$ is the best  first order Sobolev constant.
\end{theorem}

From the above Theorem \ref{BFL-Constant}, we can see that the best constant $C_n$ in (\ref{1.3}) is the same as the best sharp Sobolev constant for the first order Sobolev inequality when $n=3$.
When $n \geq 4$, Theorem \ref{Hebey-Constant} shows that the sharp constant $C_n$ in \eqref{1.3}  is strictly less than the Sobolev constant $S_{n, 1}$.
  On upper half space,
    The authors of \cite{te} subsequently offered another proof
 of $C_n < S_{n, 1}$  when $n \geq 4$. Furthermore,   they established the existence
of an extremal function for the sharp constant $C_n$.

 \medskip

 The sharp Sobolev inequalities on hyperbolic ball $\mathbb{B}^{n}$  read as follows (see Hebey \cite{heb} for $k=1$  and Liu \cite{liu} for $2\leq k<\frac{n}{2}$):
 \begin{equation}\label{1.6}
\int_{\mathbb{B}^{n}}(P_{k}u)udV\geq S_{n,k}\left(\int_{\mathbb{B}^{n}}|u|^{\frac{2n}{n-2k}}dV\right)^{\frac{n-2k}{n}},\;\;u\in
C^{\infty}_{0}(\mathbb{B}^{n}),\;\;1\leq k<\frac{n}{2},
\end{equation}
where $S_{n,k}$ is the best $k$-th order Sobolev constant.  We remark that the above sharp Sobolev inequality on hyperbolic spaces (\ref{1.6}) can also follow from the Sharp Sobolev inequality on spheres due to Beckner \cite{Be0}. An alternative proof of ({\ref{1.6}) has been given in \cite{LuYang3} (see Section 7 there).

\medskip

The following higher order Hardy-Sobolev-Maz'ya inequalities have been recently established by the authors in \cite{LuYang3} using Fourier analysis techniques on hyperbolic spaces.

\begin{theorem}\label{HSM1}
Let $2\leq k<\frac{n}{2}$ and $2<p\leq\frac{2n}{n-2k}$. There exists a positive constant $C=C(n, k, p)$ such that for each $u\in C^{\infty}_{0}(\mathbb{B}^{n})$,
\begin{equation}\label{1.11a}
\int_{\mathbb{B}^{n}}(P_{k}u)udV- \prod^{k}_{i=1}\frac{(2i-1)^{2}}{4}\int_{\mathbb{B}^{n}}u^{2}dV\geq C\left(\int_{\mathbb{B}^{n}}|u|^{p}dV\right)^{\frac{2}{p}}.
\end{equation}
\end{theorem}

This improves substantially the  Poincar\'e-Sobolev inequalities  (\ref{1.6}) for higher order derivatives on the hyperbolic spaces $\mathbb{B}^n$ by subtracting a Hardy term on the left hand side.
If $p=\frac{2n}{n-2k}$, then by (\ref{1.6}), the sharp constant  in (\ref{1.11a}) is less than or equal to the best $k$-th order Sobolev constant $S_{n, k}$.

As an application of Theorem \ref{HSM1}, we have the following Hardy-Sobolev-Maz'ya inequalities for higher order derivatives (see \cite{LuYang3}):
\begin{theorem}\label{HSM2}
Let $2\leq k<\frac{n}{2}$ and $2<p\leq\frac{2n}{n-2k}$. There exists a positive constant $C$ such that for each $u\in C^{\infty}_{0}(\mathbb{R}^{n}_{+})$,
\begin{equation}\label{1.12}
\int_{\mathbb{R}^{n}_{+}}|\nabla^{k}u|^{2}dx- \prod^{k}_{i=1}\frac{(2i-1)^{2}}{4}\int_{\mathbb{R}^{n}_{+}}\frac{u^{2}}{x^{2k}_{1}}dx\geq C\left(\int_{\mathbb{R}^{n}_{+}}x^{\gamma}_{1}|u|^{p}dx\right)^{\frac{2}{p}},
\end{equation}
where $\gamma=\frac{(n-2k)p}{2}-n$.

\medskip

In terms of the Poincar\'e ball model $\mathbb{B}^{n}$, inequality (\ref{1.11a}) can be written as follows:
\begin{equation}\label{1.14}
\int_{\mathbb{B}^{n}}|\nabla^{k}u|^{2}dx- \prod^{k}_{i=1}(2i-1)^{2}\int_{\mathbb{B}^{n}}\frac{u^{2}}{(1-|x|^{2})^{2k}}dx\geq C\left(\int_{\mathbb{B}^{n}}(1-|x|^{2})^{\gamma}|u|^{p}dx\right)^{\frac{2}{p}}.
\end{equation}
\end{theorem}

Though it was not known in general yet whether the constant $C$ on the right hand sides of the inequalities (\ref{1.11a}), (\ref{1.12}) and (\ref{1.14}) in Theorems \ref{HSM1} and \ref{HSM2} respectively is the same as the best Sobolev constant $S_{n, k}$, we have established in \cite{LuYang3}
in the case of $n=5$ and $k=2$ that the best constant for the Hardy-Sobolev-Maz'ya inequality in (\ref{1.11a}), (\ref{1.12}) and (\ref{1.14})
  coincides  with the best Sobolev constant $S_{5, 2}$. In fact, the following has been proved in \cite{LuYang3}.
\begin{theorem}\label{th1.5}
There holds,  for each $u\in C^{\infty}_{0}(\mathbb{B}^{5})$,
\begin{equation}\label{1.15}
\int_{\mathbb{B}^{5}}(P_{2}u)udV-\frac{9}{16}\int_{\mathbb{B}^{5}}u^{2}dV\geq S_{5,2}\left(\int_{\mathbb{B}^{n}}|u|^{10}dV\right)^{\frac{1}{5}}.
\end{equation}
In terms of the Poincar\'e ball model $\mathbb{B}^{5}$ and the Poincar\'e half space model $\mathbb{H}^{5}$, respectively, inequality (\ref{1.15}) is  equivalent to the following:
\begin{equation*}
\int_{\mathbb{B}^{5}}|\Delta f|^{2}dx- 9\int_{\mathbb{B}^{5}}\frac{f^{2}}{(1-|x|^{2})^{4}}dx\geq S_{5,2}\left(\int_{\mathbb{B}^{5}}|f|^{10}dx\right)^{\frac{1}{5}},\;\;f\in C^{\infty}_{0}(\mathbb{B}^{5});
\end{equation*}
\begin{equation*}
\int_{\mathbb{R}_{+}^{5}}|\Delta g|^{2}dx- \frac{9}{16}\int_{\mathbb{R}_{+}^{5}}\frac{g^{2}}{x_{1}^{4}}dx\geq S_{5,2}\left(\int_{\mathbb{R}_{+}^{5}}|g|^{10}dx\right)^{\frac{1}{5}},\;\;g\in C^{\infty}_{0}(\mathbb{R}_{+}^{5}).
\end{equation*}
\end{theorem}

 We also illustrated in \cite{LuYang3} that the best constant for the Hardy-Sobolev-Maz'ya inequality does not coincide with the best Sobolev constant  for $n=6$ and $k=2$ (see Remark 8.3 and Lemma 8.4 there).
More recently, Hong has shown in \cite{Hong} that the sharp constant of  Poincar\'e-Sobolev inequality and  the Hardy-Sobolev-Maz'ya inequality for $n=7$ and $k=3$ is also
  the same as    the Sobolev constant $S_{7, 3}$, by adapting the argument as done  in \cite{LuYang3} for the case $n=5$ and $k=2$.

\medskip

We further conjectured in \cite{LuYang3} that the sharp constant of  Poincar\'e-Sobolev inequality and  the Hardy-Sobolev-Maz'ya inequality for $n\ge 9$ and odd and $k=\frac{n-1}{2}$ also
  coincides  with the Sobolev constant $S_{n, \frac{n-1}{2}}$.

One of the our main results in this paper is to confirm this conjecture and give a unified proof for $ k=\frac{n-1}{2}$ and all $n\ge 5$ odd.

\begin{theorem} \label{th1.1} Let $n\geq5$ be odd. There holds,  for each $u\in C^{\infty}_{0}(\mathbb{B}^{n})$,
\begin{equation*}
\int_{\mathbb{B}^{n}}(P_{(n-1)/2}u)udV- \prod^{(n-1)/2}_{j=1}\frac{(2j-1)^{2}}{4}\int_{\mathbb{B}^{n}}u^{2}dV\geq S_{n,(n-1)/2}\left(\int_{\mathbb{B}^{n}}|u|^{2n}dV\right)^{\frac{1}{n}}.
\end{equation*}
In terms of  the Poincar\'e ball model and the Poincar\'e half space model respectively, the inequality above is  equivalent to the following:
\begin{equation*}
\int_{\mathbb{B}^{n}}|\nabla^{\frac{n-1}{2}}f|^{2}dx- \prod^{(n-1)/2}_{j=1}(2j-1)^{2}\int_{\mathbb{B}^{n}}\frac{f^{2}}{(1-|x|^{2})^{n-1}}dx\geq S_{n,(n-1)/2}\left(\int_{\mathbb{B}^{n}}|f|^{2n}dx\right)^{\frac{1}{n}},\;\;f\in C^{\infty}_{0}(\mathbb{B}^{n});
\end{equation*}
\begin{equation*}
\int_{\mathbb{R}^{n}_{+}}|\nabla^{\frac{n-1}{2}}g|^{2}dx- \prod^{(n-1)/2}_{j=1}\frac{(2j-1)^{2}}{4}\int_{\mathbb{R}^{n}_{+}}\frac{g^{2}}{x^{n-1}_{1}}dx\geq S_{n,(n-1)/2}\left(\int_{\mathbb{R}_{+}^{n}}|g|^{2n}dx\right)^{\frac{1}{n}},\;\;g\in C^{\infty}_{0}(\mathbb{R}_{+}^{n}).
\end{equation*}
\end{theorem}

Our second main result is:
\begin{theorem} \label{th1.2}
Let $2\leq k<\frac{n}{2}$.  Suppose that there exists $\lambda\in\mathbb{R}$ such that for
any $u\in C^{\infty}_{0}(\mathbb{B}^{n})$,
\[
\int_{\mathbb{B}^{n}}(P_{k}u)udV+\lambda \int_{\mathbb{B}^{n}}u^{2}dV\geq S_{n,k}
\left(\int_{\mathbb{B}^{n}}|u|^{\frac{2n}{n-2k}}dV\right)^{\frac{n-2k}{n}},
\]
where $S_{n,k}$ is the best $k$-th order Sobolev constant. Then we have the following:

(1). If $n\geq 4k$, then $\lambda\geq0$.

(2). If $2k+2\leq n< 4k$, then
\begin{equation}\label{1.1}
\begin{split}
\lambda
\geq-\frac{\Gamma(n/2)\Gamma(k)\sum\limits^{k-1}_{j=0}
\frac{\Gamma(j+\frac{n-2k}{2})}{\Gamma(j+1)\Gamma(\frac{n-2k}{2})}}{2^{\frac{n+2k}{2}}\Gamma(\frac{n-2k}{2})\int^{1}_{0}[r^{2k-n}-\sum\limits^{k-1}_{j=0}
\frac{\Gamma(j+\frac{n-2k}{2})}{\Gamma(j+1)\Gamma(\frac{n-2k}{2})}
\left(1-r^{2}\right)^{j}]^{2}\frac{r^{n-1}dr}{(1-r^{2})^{2k}}}
.
\end{split}
\end{equation}
\end{theorem}

\begin{remark}
By Theorem \ref{th1.2}, we have shown that
if $n\geq 4k$, then the sharp constant of $2k$-th order Hardy-Sobolev-Maz'ya inequality is  strictly less than $S_{n,k}$.
 In some cases of $2k+2\leq n< 4k$, we can also show that the sharp constant of $2k$-th order Hardy-Sobolev-Maz'ya inequality is  also strictly less than $S_{n,k}$. For instance, if
$n=2k+2$, we can compute the right hand of inequality (\ref{1.1}) is equal to $-\frac{(k-1)(k!)^{2}}{2^{2k}}$,
which is strictly greater than $-\prod\limits^{k}_{j=1}\frac{(2j-1)^{2}}{4}$.  Similar results hold in the rest of the cases $n$, i. e. $2k+2<n<4k$.
We leave it to the  reader.
\end{remark}

\begin{remark}
When $n=2k+1$, i.e., $k=\frac{n-1}{2}$ and $n$ is odd, as shown in Theorem \ref{th1.1}, the sharp constants for the Sobolev inequality and the Hardy-Sobolev-Maz'ya inequality are the same. When $n$ is even, and $k=\frac{n}{2}$, this is the borderline case,
the Hardy-Adams inequalities hold as shown in  \cite{LiLuYang1}, \cite{LiLuYang2},   \cite{LuYang2} (see also \cite{LuYang1}, \cite{wy} for the Hardy-Trudinger-Moser inequality).
\end{remark}

Our next main results in this paper are the estimates of Green's functions of the Paneitz and GJMS operators on hyperbolic spaces
which are of independent interests.
It is well known   that Green's function of the second order operator $P_{1}$
\begin{equation*}
P^{-1}_{1}=\frac{1}{\gamma_{n}(2)}\left[\left(\frac{1}{2\sinh\frac{\rho}{2}}\right)^{n-2}-
\left(\frac{1}{2\cosh\frac{\rho}{2}}\right)^{n-2}\right],
\end{equation*}
where  $\gamma_{n}(\alpha)$ is defined by
\begin{equation}\label{4.13}
\begin{split}
\gamma_{n}(\alpha)=\pi^{n/2}2^{\alpha}\frac{\Gamma(\alpha/2)}{\Gamma(\frac{n-\alpha}{2})},\;\;\;\;0<\alpha<n
\end{split}
\end{equation}
and $\rho=\log \frac{1+|x|}{1-|x|}$ is the hyperbolic distance from $x\in \mathbb{B}^n$ to the origin.

We will  establish the following  precise expressions of the Green functions for the Paneitz operator $P_2$ and higher order GJMS operators $P_k$, which are of independent interest.
\begin{theorem} \label{th1.3}
Let $1<k<n/2$.
The Green's function of $P_{k}$ satisfies:
\begin{itemize}
  \item if $n=2m$, then
  \begin{equation*}
\begin{split}
P^{-1}_{k}(\rho)
=&\frac{4^{k-1}}{\gamma_{2m}(2k)(\sinh\rho)^{2m-2}}\sum^{m-1-k}_{j=0}\frac{\Gamma(m)}{\Gamma(j+1)\Gamma(m-j)}
\left(\sinh\frac{\rho}{2}\right)^{2j+2k-2};
\end{split}
\end{equation*}

  \item if $n=2m-1$, then
    \begin{equation*}
\begin{split}
 P^{-1}_{k}(\rho)= \frac{4^{k-1}\sqrt{2}}{\gamma_{2m}(2k)}\sum^{m-1-k}_{j=0}\frac{\Gamma(m)}{\Gamma(j+1)\Gamma(m-j)}
\int^{\infty}_{\rho}
\left(\sinh\frac{r}{2}\right)^{2j+2k-2}\frac{(\sinh r)^{3-2m}}{\sqrt{\cosh r-\cosh\rho}}dr.
\end{split}
\end{equation*}
\end{itemize}
Furthermore, we have
\begin{equation*}
P^{-1}_{k}(\rho)\leq\frac{1}{\gamma_{n}(2k)}\left[\left(\frac{1}{2\sinh\frac{\rho}{2}}\right)^{n-2k}-
\left(\frac{1}{2\cosh\frac{\rho}{2}}\right)^{n-2k}\right],\;\;\rho>0.
\end{equation*}
\end{theorem}

The organization of this paper is as follows. In Section 2, we will review some Fourier analysis on hyperbolic spaces that will be needed in the subsequent sections. Section 3 gives a precise expression of the Green function of the operator  $(-\frac{(n-1)^{2}}{4}-\Delta_{\mathbb{H}})^{-1}$ on hyperbolic spaces using the Legendre function of second type $Q^{\mu}_{\nu}(z)$.
In Section 4, when $n$ is odd and $n=2m+1$ or when and $0\leq k\leq m-1$,  we   give precise  expressions for the Green functions of the operators
$(k^{2}-(n-1)^{2}/4-\Delta_{\mathbb{H}})^{-1}$   and
$\left[\prod^{l-1}_{j=0}((k+j)^{2}-(n-1)^{2}/4-\Delta_{\mathbb{H}})\right]^{-1}$ for $l\in \{1, 2, \cdots, m\}$ and $k\in \{0, 1, \cdots, m-l\}$. The proof of Theorem \ref{th1.1} is then given using these Green's functions and Fourier analysis on hyperbolic spaces.
Section 5 gives the proof of Theorem \ref{th1.2}. In Section 6, we provide the precise expressions of Green functions of  the Paneitz and GJMS operators $P_k$ on hyperbolic spaces and offer an optimal pointwise bound of the Green functions. The proof of Theorem \ref{th1.3} is given in this section.

\section{Notations and preliminaries}
We begin by quoting some preliminary facts on Fourier analysis on hyperbolic spaces which will be needed in
the sequel and  refer to \cite{ah,an,anj, CS, DR, he,he2,hu, Ionescu, Ionescu1, liup, Stein1, ter,Wolf} for more information about this subject.

\subsection{The half space model $\mathbb{H}^{n}$}

It is given by $\mathbb{R}_{+}\times\mathbb{R}^{n-1}=\{(x_{1},\cdots,x_{n}): x_{1}>0\}$ equipped with the Riemannian metric
$ds^{2}=\frac{dx_{1}^{2}+\cdots+dx_{n}^{2}}{x^{2}_{1}}$. The induced Riemannian measure can be written as $dV=\frac{dx}{x^{n}_{1}}$, where $dx$ is the Lebesgue measure on
$\mathbb{R}^{n}$.
The hyperbolic gradient is $\nabla_{\mathbb{H}}=x_{1}\nabla$ and the
Laplace-Beltrami operator on $\mathbb{H}^{n}$ is given by
\begin{equation}\label{2.1}
\Delta_{\mathbb{H}}=x^{2}_{1}\Delta-(n-2)x_{1}\frac{\partial}{\partial x_{1}},
\end{equation}
where $\Delta=\sum^{n}_{i=1}\frac{\partial^{2}}{\partial
x^{2}_{i}}$ is the Laplace operator on $\mathbb{R}^{n}$.

\subsection{The ball model $\mathbb{B}^{n}$}
Recall that
it is given by the unit ball
\[\mathbb{B}^{n}=\{x=(x_{1},\cdots,x_{n})\in \mathbb{R}^{n}| |x|<1\}\]
equipped with the usual Poincar\'e metric
\[
ds^{2}=\frac{4(dx^{2}_{1}+\cdots+dx^{2}_{n})}{(1-|x|^{2})^{2}}.
\]
The hyperbolic gradient is $\nabla_{\mathbb{H}}=\frac{1-|x|^{2}}{2}\nabla$ and
the
Laplace-Beltrami operator is given by
\[
\Delta_{\mathbb{H}}=\frac{1-|x|^{2}}{4}\left\{(1-|x|^{2})\sum^{n}_{i=1}\frac{\partial^{2}}{\partial
x^{2}_{i}}+2(n-2)\sum^{n}_{i=1}x_{i}\frac{\partial}{\partial
x_{i}}\right\}.
\]
Furthermore,  the half space model $\mathbb{H}^{n}$ and the ball model $\mathbb{B}^{n}$ are  equivalent.

\subsection{M\"obius
transformations}
For each $a\in\mathbb{B}^{n}$, we define the  M\"obius
transformations $T_{a}$ by (see e.g. \cite{ah,hu})
\[
T_{a}(x)=\frac{|x-a|^{2}a-(1-|a|^{2})(x-a)}{1-2x\cdot
a+|x|^{2}|a|^{2}},
\]
where $x\cdot a=x_{1}a_{1}+x_{2}a_{2}+\cdots +x_{n}a_{n}$ denotes
the  scalar product in $\mathbb{R}^{n}$. It is known that the measure on $\mathbb{B}^{n}$ is invariant with respect to the M\"obius
transformations.

Using the M\"obius transformations, we can define the
distance  from $x$ to $y$ in $\mathbb{B}^{n}$ as follows
\begin{equation*}
\rho(x,y)=\rho(T_{x}(y))=\rho(T_{y}(x))=\log\frac{1+|T_{y}(x)|}{1-|T_{y}(x)|}.
\end{equation*}
Also using the M\"obius transformations,  we can define the convolution of
measurable functions $f$ and $g$ on $\mathbb{B}^{n}$ by
\begin{equation}\label{2.3}
(f\ast g)(x)=\int_{\mathbb{B}^{n}}f(y)g(T_{x}(y))dV(y)
\end{equation}
provided this integral exists. It is easy to check that
\[
f\ast g= g\ast f.
\]
Furthermore, if $g$ is radial, i.e. $g=g(\rho)$, then
\begin{equation}\label{2.4}
  (f\ast g)\ast h= f\ast (g\ast h)
\end{equation}
provided $f,g,h\in L^{1}(\mathbb{B}^{n})$

\subsection{The Fourier transform on hyperbolic spaces}
In term of the ball model, we define the Fourier transform on hyperbolic spaces as follows. Set
\[
e_{\lambda,\zeta}(x)=\left(\frac{\sqrt{1-|x|^{2}}}{|x-\zeta|}\right)^{n-1+i\lambda}, \;\; x\in \mathbb{B}^{n},\;\;\lambda\in\mathbb{R},\;\;\zeta\in\mathbb{S}^{n-1}.
\]
The Fourier transform of a function  $f$  on $\mathbb{B}^{n}$ can be defined as
\[
\widehat{f}(\lambda,\zeta)=\int_{\mathbb{B}^{n}} f(x)e_{-\lambda,\zeta}(x)dV
\]
provided this integral exists. If $g\in C^{\infty}_{0}(\mathbb{B}^{n})$ is radial, then
 $$\widehat{(f\ast g)}=\widehat{f}\cdot\widehat{g}.$$
Moreover,
the following inversion formula holds for $f\in C^{\infty}_{0}(\mathbb{B}^{n})$:
\[
f(x)=D_{n}\int^{+\infty}_{-\infty}\int_{\mathbb{S}^{n-1}} \widehat{f}(\lambda,\zeta)e_{\lambda,\zeta}(x)|\mathfrak{c}(\lambda)|^{-2}d\lambda d\sigma(\varsigma),
\]
where $D_{n}=\frac{1}{2^{3-n}\pi |\mathbb{S}^{n-1}|}$ and $\mathfrak{c}(\lambda)$ is the  Harish-Chandra $\mathfrak{c}$-function given by
\[
\mathfrak{c}(\lambda)=\frac{2^{n-1-i\lambda}\Gamma(n/2)\Gamma(i\lambda)}{\Gamma(\frac{n-1+i\lambda}{2})\Gamma(\frac{1+i\lambda}{2})}.
\]
Similarly, there holds the Plancherel formula:
\begin{equation}\label{2.5}
\int_{\mathbb{B}^{n}}|f(x)|^{2}dV=D_{n}\int^{+\infty}_{-\infty}\int_{\mathbb{S}^{n-1}}|\widehat{f}(\lambda,\zeta)|^{2}|\mathfrak{c}(\lambda)|^{-2}d\lambda d\sigma(\varsigma).
\end{equation}

Since $e_{\lambda,\zeta}(x)$ is an eigenfunction of $\Delta_{\mathbb{H}}$ with eigenvalue $-\frac{(n-1)^{2}+\lambda^{2}}{4}$, it is easy to check that, for
$f\in C^{\infty}_{0}(\mathbb{B}^{n})$,
\[
\widehat{\Delta_{\mathbb{H}}f}(\lambda,\zeta)=-\frac{(n-1)^{2}+\lambda^{2}}{4}\widehat{f}(\lambda,\zeta).
\]
Therefore, in analogy with the Euclidean setting, we define the fractional
Laplacian on hyperbolic space as follows:
\begin{equation}\label{2.6}
\widehat{(-\Delta_{\mathbb{H}})^{\gamma}f}(\lambda,\zeta)=\left(\frac{(n-1)^{2}+\lambda^{2}}{4}\right)^{\gamma}\widehat{f}(\lambda,\zeta),\;\;\gamma\in \mathbb{R}.
\end{equation}

\section{Precise expression and estimate for Green's function of $(-\frac{(n-1)^{2}}{4}-\Delta_{\mathbb{H}})^{-1}$}

 In this section, we give a precise  expression and estimate of
 the kernel of Green's function of the operator $(-(n-1)^{2}/4-\Delta_{\mathbb{H}})^{-1}$, where $(n-1)^{2}/4$ is the spectral gap of the Laplacian $-\Delta_{\mathbb{H}}$ on the hyperbolic ball $\mathbb{B}^n$.
  We will need some facts about the Legendre function of second type $Q^{\mu}_{\nu}(z)$, which is defined by (see \cite{er})
\begin{equation}\label{b2.1}
\begin{split}
Q^{\mu}_{\nu}(z)
=&e^{i(\pi\mu)}2^{-\nu-1}\frac{\Gamma(\nu+\mu+1)}{\Gamma(\nu+1)}
(z^{2}-1)^{-\mu/2}\int^{\pi}_{0}(z+\cos t)^{\mu-\nu-1}
(\sin t)^{2\nu+1}dt,\\
& Re\nu>-1,\;\;Re(\nu+\mu+1)>0.
\end{split}
\end{equation}
It is known that (see \cite{er})
\begin{equation}\label{b2.2}
\begin{split}
Q^{\mu}_{\nu}(\cosh\rho)
=&e^{i(\pi\mu)}\frac{\sqrt{\pi}}{\sqrt{2}}\frac{\sinh^{\mu}\rho}{
\Gamma(1/2-\mu)}\int^{\infty}_{\rho}e^{-(\nu+\frac{1}{2})r}(\cosh r-\cosh\rho)^{-\mu-1/2}dr,\\
&\rho>0,\;\; Re(\nu+\mu+1)>0,\;\;Re\mu<1/2.
\end{split}
\end{equation}
and (see \cite{gr})
\begin{equation}\label{b2.3}
\begin{split}
\int^{\infty}_{u}(x^{2}-1)^{\frac{1}{2}\lambda}(x-u)^{\mu-1}Q^{-\lambda}_{\nu}(x)dx=&\Gamma(\mu)e^{\mu\pi i}(u^{2}-1)^{\frac{1}{2}\lambda+\frac{1}{2}\mu}
Q^{-\lambda-\mu}_{\nu}(u),\\
&|\arg(u-1)|<\pi,\;\;0<Re\mu<1+Re(\nu-\lambda).
\end{split}
\end{equation}
Setting $z=\cosh\rho$ and using (\ref{b2.1}), we have
\begin{equation}\label{b2.5}
\begin{split}
Q^{\mu}_{\nu}(\cosh\rho)
=e^{i(\pi\mu)}2^{-\nu-1}\frac{\Gamma(\nu+\mu+1)}{\Gamma(\nu+1)}
\sinh^{-\mu}\rho\int^{\pi}_{0}(\cosh\rho+\cos t)^{\mu-\nu-1}
(\sin t)^{2\nu+1}dt.
\end{split}
\end{equation}
Setting $u=\cosh\rho$ and using (\ref{b2.3}), we have
\begin{equation}\label{bb2.3}
\begin{split}
\int^{\infty}_{\rho}\frac{(\sinh r)^{\lambda+1}}{(\cosh r-\cosh\rho)^{1-\mu}}Q^{-\lambda}_{\nu}(\cosh r)dr=&\Gamma(\mu)e^{\mu\pi i}(\sinh\rho)^{\lambda+\mu}
Q^{-\lambda-\mu}_{\nu}(\cosh\rho).
\end{split}
\end{equation}

Let $n\geq2$. From now on, we will denote by $e^{t\Delta_{\mathbb{H}}}$ the heat kernel on $\mathbb{B}^{n}$. It is well known that $e^{t\Delta_{\mathbb{H}}}$ depends only on $t$ and $\rho(x,y)$. In fact, $e^{t\Delta_{\mathbb{H}}}$  is given explicitly by the following formulas (see e.g. \cite{d,gn}):

\begin{itemize}
  \item If $n=2m$, then
  \begin{equation}\label{3.1}
\begin{split}
e^{t\Delta_{\mathbb{H}}}=&(2\pi)^{-\frac{n+1}{2}}t^{-\frac{1}{2}}e^{-\frac{(n-1)^{2}}{4}t}
\int^{+\infty}_{\rho}\frac{\sinh r}{\sqrt{\cosh r-\cosh\rho}}\left(-\frac{1}{\sinh r}\frac{\partial}{\partial r}\right)^{m}e^{-\frac{r^{2}}{4t}}dr\\
=&\frac{1}{2(2\pi)^{\frac{n+1}{2}}}t^{-\frac{3}{2}}\int^{+\infty}_{\rho}\frac{\sinh r}{\sqrt{\cosh r-\cosh\rho}}\left(-\frac{1}{\sinh r}\frac{\partial}{\partial r}\right)^{m-1}\left(\frac{r}{\sinh r} e^{-\frac{r^{2}}{4t}}\right)dr;
\end{split}
\end{equation}
  \item If $n=2m+1$, then
   \begin{equation}\label{3.2}
  \begin{split}
e^{t\Delta_{\mathbb{H}}}=&2^{-m-1}\pi^{-m-1/2}t^{-\frac{1}{2}}e^{-\frac{(n-1)^{2}}{4}t}\left(-\frac{1}{\sinh \rho}\frac{\partial}{\partial \rho}\right)^{m}e^{-\frac{\rho^{2}}{4t}}\\
=&2^{-m-2}\pi^{-m-1/2}t^{-\frac{3}{2}}e^{-\frac{(n-1)^{2}}{4}t}\left(-\frac{1}{\sinh \rho}\frac{\partial}{\partial \rho}\right)^{m-1}\left(\frac{\rho}{\sinh \rho} e^{-\frac{\rho^{2}}{4t}}\right).
\end{split}
\end{equation}

\end{itemize}

  An explicit expression of Green's function $(\lambda-\Delta_{\mathbb{H}})^{-1}$ with $\lambda>-\frac{(n-1)^{2}}{4}$ is given by (see \cite{mat} for $\lambda\geq0$ and \cite{li} for $\lambda>-\frac{(n-1)^{2}}{4}$)
\begin{equation}\label{3.3}
\begin{split}
(\lambda-\Delta_{\mathbb{H}})^{-1}
=&(2\pi)^{-\frac{n}{2}}(\sinh\rho)^{-\frac{n-2}{2}}e^{-\frac{(n-2)\pi}{2}i}Q^{\frac{n-2}{2}}_{\theta_{n}(\lambda)}(\cosh\rho), \;\; n\geq3.
\end{split}
\end{equation}
where
\begin{equation}\label{3.4}
\begin{split}
\theta_{n}(\lambda)=\sqrt{\lambda+\frac{(n-1)^{2}}{4}}-\frac{1}{2}.
\end{split}
\end{equation}
We shall show that (\ref{3.3}) is also valid for $\lambda=-\frac{(n-1)^{2}}{4}$. Firstly, we prove the following Lemma:
\begin{lemma}\label{lm3.1} Let $m$ be a positive integer. Then we have
\begin{equation}\label{b2.7}
\begin{split}
\left(-\frac{1}{\sinh \rho}\frac{\partial}{\partial \rho}\right)^{m-1}\frac{1}{\sinh \rho}
=&\frac{\Gamma(m)}{\pi}\cdot\frac{1}{(\sinh\rho)^{2m-1}}\int^{\pi}_{0}(\cosh\rho+\cos t)^{m-1}
dt,\;\;\rho>0.
\end{split}
\end{equation}
\end{lemma}
\begin{proof}
We shall prove it by induction. It is easy to check (\ref{b2.7}) is valid for
$m=1$ and $m=2$.
Assume (\ref{b2.7}) is valid for $m=k\geq2$. Then for $m=k+1$, we have
\begin{equation*}
\begin{split}
\left(-\frac{1}{\sinh \rho}\frac{\partial}{\partial \rho}\right)^{k}\frac{1}{\sinh \rho}
=&\frac{\Gamma(k)}{\pi}\left(-\frac{1}{\sinh \rho}\frac{\partial}{\partial \rho}\right)\frac{1}{(\sinh\rho)^{2k-1}}\int^{\pi}_{0}(\cosh\rho+\cos t)^{k-1}
dt\\
=&\frac{\Gamma(k)}{\pi}\frac{1}{(\sinh\rho)^{2k+1}}\left[(2k-1)\cosh\rho\int^{\pi}_{0}(\cosh\rho+\cos t)^{k-1}
dt-\right.\\
&\left.(k-1)\sinh^{2}\rho\int^{\pi}_{0}(\cosh\rho+\cos t)^{k-2}
dt\right].
\end{split}
\end{equation*}
To finish the proof, it is enough to show
\begin{equation*}
\begin{split}
&(2k-1)\cosh\rho\int^{\pi}_{0}(\cosh\rho+\cos t)^{k-1}
dt-(k-1)\sinh^{2}\rho\int^{\pi}_{0}(\cosh\rho+\cos t)^{k-2}
dt\\
=&k\int^{\pi}_{0}(\cosh\rho+\cos t)^{k}
dt.
\end{split}
\end{equation*}
In fact,
\begin{equation*}
\begin{split}
&(2k-1)\cosh\rho\int^{\pi}_{0}(\cosh\rho+\cos t)^{k-1}
dt-(k-1)\sinh^{2}\rho\int^{\pi}_{0}(\cosh\rho+\cos t)^{k-2}
dt\\
&-k\int^{\pi}_{0}(\cosh\rho+\cos t)^{k}
dt\\
=&\int^{\pi}_{0}(\cosh\rho+\cos t)^{k-2}\left[(2k-1)\cosh\rho(\cosh\rho+\cos t)-(k-1)\sinh^{2}\rho-k(\cosh\rho+\cos t)^{2}\right]
dt\\
=&\int^{\pi}_{0}(\cosh\rho+\cos t)^{k-2}\left[(k-1)\sin^{2}t-\cos t(\cosh\rho+\cos t)\right]dt\\
=&(k-1)\int^{\pi}_{0}(\cosh\rho+\cos t)^{k-2}\sin^{2}tdt-\int^{\pi}_{0}(\cosh\rho+\cos t)^{k-1}\cos tdt
dt=0.
\end{split}
\end{equation*}
This completes the proof of Lemma \ref{lm3.1}.
\end{proof}

Now we  show that (\ref{3.3}) is  valid for $\lambda=-\frac{(n-1)^{2}}{4}$.
\begin{lemma}\label{lm3.2b}  There holds, for $n\geq3$ and $\rho>0$,
\begin{equation}\label{b2.8}
\begin{split}
(-(n-1)^{2}/4-\Delta_{\mathbb{H}})^{-1}
=&(2\pi)^{-\frac{n}{2}}(\sinh\rho)^{-\frac{n-2}{2}}e^{-\frac{(n-2)\pi}{2}i}Q^{\frac{n-2}{2}}_{-1/2}(\cosh\rho).
\end{split}
\end{equation}
\end{lemma}
\begin{proof}
If $n=2m+1$, then by (\ref{3.2}), (\ref{b2.7}) and (\ref{b2.5}), we have
\begin{equation*}
\begin{split}
\left(-(n-1)^{2}/4-\Delta_{\mathbb{H}}\right)^{-1}=&\int^{\infty}_{0}e^{t(\Delta_{\mathbb{H}}+\frac{(n-1)^{2}}{4})}dt\\
=&2^{-m-2}\pi^{-m-1/2}\left(-\frac{1}{\sinh \rho}\frac{\partial}{\partial \rho}\right)^{m-1}\left(\frac{\rho}{\sinh \rho} \int^{\infty}_{0}t^{-\frac{3}{2}}e^{-\frac{\rho^{2}}{4t}}dt\right)\\
=&2^{-m-1}\pi^{-m}\left(-\frac{1}{\sinh \rho}\frac{\partial}{\partial \rho}\right)^{m-1}\frac{1}{\sinh \rho}\\
=&\frac{\Gamma(m)}{(2\pi)^{m+1}}\cdot\frac{1}{(\sinh\rho)^{2m-1}}\int^{\pi}_{0}(\cosh\rho+\cos t)^{m-1}
dt\\
=&\frac{\Gamma(\frac{n-1}{2})}{(2\pi)^{\frac{n+1}{2}}}\cdot\frac{1}{(\sinh\rho)^{n-2}}\int^{\pi}_{0}(\cosh\rho+\cos t)^{\frac{n-3}{2}}
dt\\
=&(2\pi)^{-\frac{n}{2}}(\sinh\rho)^{-\frac{n-2}{2}}e^{-\frac{(n-2)\pi}{2}i}Q^{(n-2)/2}_{-1/2}(\cosh\rho).
\end{split}
\end{equation*}

If $n=2m$,  then by (\ref{3.1}), (\ref{b2.7}), (\ref{b2.5}) and (\ref{bb2.3}), we have
\begin{equation*}
\begin{split}
&(-\Delta_{\mathbb{H}}-(n-1)^{2}/4)^{-1}=\int^{\infty}_{0}e^{t(\Delta_{\mathbb{H}}+\frac{(n-1)^{2}}{4})}dt\\
=&\frac{1}{2(2\pi)^{\frac{n+1}{2}}}\int^{+\infty}_{\rho}\frac{\sinh r}{\sqrt{\cosh r-\cosh\rho}}\left(-\frac{1}{\sinh r}\frac{\partial}{\partial r}\right)^{m-1}\left(\frac{r}{\sinh r} \int^{\infty}_{0}t^{-\frac{3}{2}}e^{-\frac{r^{2}}{4t}}dt\right)dr\\
=&\frac{\sqrt{\pi}}{(2\pi)^{\frac{n+1}{2}}}\int^{+\infty}_{\rho}\frac{\sinh r}{\sqrt{\cosh r-\cosh\rho}}\left(-\frac{1}{\sinh r}\frac{\partial}{\partial r}\right)^{m-1}\frac{1}{\sinh r} dr\\
=&\frac{\sqrt{\pi}}{(2\pi)^{\frac{n+1}{2}}}\int^{+\infty}_{\rho}\frac{\sinh r}{\sqrt{\cosh r-\cosh\rho}}\left[\frac{\Gamma(m)}{\pi}\cdot\frac{1}{(\sinh r)^{2m-1}}\int^{\pi}_{0}(\cosh r+\cos t)^{m-1}
dt\right] dr\\
=&\frac{\sqrt{\pi}}{(2\pi)^{\frac{n+1}{2}}}\int^{+\infty}_{\rho}\frac{(\sinh r)^{-\frac{2m-3}{2}}}{\sqrt{\cosh r-\cosh\rho}}\sqrt{\frac{2}{\pi}}e^{-\frac{(2m-1)\pi}{2}i}Q^{(2m-1)/2}_{-1/2}(\cosh r) dr\\
=&\frac{\sqrt{2}}{(2\pi)^{\frac{n+1}{2}}}e^{-\frac{(2m-1)\pi}{2}i}\int^{+\infty}_{\rho}\frac{(\sinh r)^{-\frac{2m-3}{2}}}{\sqrt{\cosh r-\cosh\rho}}Q^{(2m-1)/2}_{-1/2}(\cosh r) dr\\
=&\frac{\sqrt{2}}{(2\pi)^{\frac{n+1}{2}}}e^{-\frac{(2m-1)\pi}{2}i}\cdot \Gamma(1/2)e^{\frac{\pi i}{2}}(\sinh\rho)^{1-m}
Q^{m-1}_{-1/2}(\cosh \rho)\\
=&(2\pi)^{-\frac{n}{2}}(\sinh\rho)^{-\frac{n-2}{2}}e^{-\frac{(n-2)\pi}{2}i}Q^{(n-2)/2}_{-1/2}(\cosh\rho).
\end{split}
\end{equation*}
The proof of Lemma \ref{lm3.2b}
is thereby completed.
\end{proof}

Therefore, for $n\geq3$ and $\lambda= \nu^{2}-(n-1)^{2}/4$ with $\nu\geq0$, we have, by Lemma \ref{lm3.2b} and (\ref{b2.5}),
\begin{equation}\label{b2.9}
\begin{split}
&(\nu^{2}-(n-1)^{2}/4-\Delta_{\mathbb{H}})^{-1}=(2\pi)^{-\frac{n}{2}}(\sinh\rho)^{-\frac{n-2}{2}}
e^{-\frac{(n-2)\pi}{2}i}Q^{\frac{n-2}{2}}_{\nu-\frac{1}{2}}(\cosh\rho)\\
=&\frac{(2\pi)^{-\frac{n}{2}}\Gamma(\frac{n-1}{2}+\nu)}{2^{\nu+\frac{1}{2}}
\Gamma(\nu+\frac{1}{2})}(\sinh\rho)^{2-n}\int^{\pi}_{0}(\cosh\rho+\cos t)^{\frac{n-3}{2}-\nu}
(\sin t)^{2\nu}dt.
\end{split}
\end{equation}

\section{More Green's function estimates and Proof of Theorem \ref{th1.1}}

 In this section,
 when $n=2m+1$ and $0\leq k\leq m-1$,  we will  give precise  expressions for the Green functions of the operators
$(k^{2}-(n-1)^{2}/4-\Delta_{\mathbb{H}})^{-1}$   and
$\left[\prod^{l-1}_{j=0}((k+j)^{2}-(n-1)^{2}/4-\Delta_{\mathbb{H}})\right]^{-1}$ for $l\in \{1, 2, \cdots, m\}$ and $k\in \{0, 1, \cdots, m-l\}$.
These estimates play an important role in our proof of Theorem \ref{th1.1}.

\begin{lemma}\label{lm4.1}If $n=2m+1$ and $0\leq k\leq m-1$, then
\begin{equation}\label{4.1}
\begin{split}
&(k^{2}-(n-1)^{2}/4-\Delta_{\mathbb{H}})^{-1}\\
=&\frac{1}{4\pi^{m+\frac{1}{2}}\left(\sinh \rho\right)^{2m-1}}\sum^{m-k-1}_{j=0}\frac{\Gamma(m+k)\Gamma(m-k)\Gamma(m-j-\frac{1}{2})}{\Gamma(j+1)\Gamma(m-k-j)\Gamma(m+k-j)}
\left(\sinh\frac{\rho}{2}\right)^{2j}.
\end{split}
\end{equation}
\end{lemma}
\begin{proof}
If $n=2m+1$ and $0\leq k\leq m-1$, then by (\ref{b2.9}) and binomial theorem, we have
\begin{equation}\label{4.2}
\begin{split}
&(k^{2}-(n-1)^{2}/4-\Delta_{\mathbb{H}})^{-1}\\
=&\frac{(2\pi)^{-\frac{n}{2}}\Gamma(m+k)}{2^{k+\frac{1}{2}}
\Gamma(k+\frac{1}{2})}(\sinh\rho)^{1-2m}\int^{\pi}_{0}(\cosh\rho+\cos t)^{m-1-k}
(\sin t)^{2k}dt\\
=&\frac{(2\pi)^{-\frac{n}{2}}\Gamma(m+k)}{2^{k+\frac{1}{2}}
\Gamma(k+\frac{1}{2})}(\sinh\rho)^{1-2m}\int^{\pi}_{0}(2\sinh^{2}\frac{\rho}{2}+1+\cos t)^{m-1-k}
(\sin t)^{2k}dt\\
=&\frac{(2\pi)^{-\frac{n}{2}}\Gamma(m+k)}{2^{k+\frac{1}{2}}
\Gamma(k+\frac{1}{2})}(\sinh\rho)^{1-2m}\sum^{m-1-k}_{j=0}C^{j}_{m-1-k}(2\sinh^{2}\frac{\rho}{2})^{j}
\int^{\pi}_{0}(1+\cos t)^{m-1-k-j}
(\sin t)^{2k}dt,
\end{split}
\end{equation}
where
\begin{equation}\label{4.3}
\begin{split}
C^{j}_{m-1-k}=\frac{(m-1-k)!}{j!(m-1-k-j)!}=\frac{\Gamma(m-k)}{\Gamma(j+1)\Gamma(m-k-j)}.
\end{split}
\end{equation}
Notice that, for $p>-\frac{q+1}{2}$ and $q>-1$,
\begin{equation}\label{4.4}
\begin{split}
\int^{\pi}_{0}(1+\cos t)^{p}(\sin t)^{q}dt=&2^{p+q}\int^{\pi}_{0}(\cos t/2)^{2p+q}(\sin t/2)^{q}dt\\
=&2^{p+q}\textbf{B}\left(p+\frac{q+1}{2},\frac{q+1}{2}\right)\\
=&2^{p+q}\frac{\Gamma(p+\frac{q+1}{2})\Gamma(\frac{q+1}{2})}{\Gamma(p+q+1)}.
\end{split}
\end{equation}
We have, by (\ref{4.2})-(\ref{4.4}),
\begin{equation}\label{4.5}
\begin{split}
&(k^{2}-(n-1)^{2}/4-\Delta_{\mathbb{H}})^{-1}\\
=&\frac{(2\pi)^{-\frac{n}{2}}\Gamma(m+k)}{2^{k+\frac{1}{2}}
\Gamma(k+\frac{1}{2})}(\sinh\rho)^{1-2m}\cdot\\
&\sum^{m-1-k}_{j=0}\frac{\Gamma(m-k)}{\Gamma(j+1)\Gamma(m-k-j)}((2\sinh^{2}\frac{\rho}{2})^{j}
2^{m+k-1-j}\frac{\Gamma(m-j-\frac{1}{2})\Gamma(k+\frac{1}{2})}{\Gamma(m+k-j)}\\
=&\frac{1}{4\pi^{m+\frac{1}{2}}\left(\sinh \rho\right)^{2m-1}}\sum^{m-k-1}_{j=0}\frac{\Gamma(m+k)\Gamma(m-k)\Gamma(m-j-\frac{1}{2})}{\Gamma(j+1)\Gamma(m-k-j)\Gamma(m+k-j)}
\left(\sinh\frac{\rho}{2}\right)^{2j}.
\end{split}
\end{equation}
\end{proof}

In general, we have the following:
\begin{lemma}\label{lm4.2}Let $n=2m+1$ and $l\in\{1,2,\cdots,m\}$. We have, for $k\in\{0,1,\cdots, m-l\}$ and $\rho>0$,
\begin{equation}\label{4.6}
\begin{split}
&\left[\prod^{l-1}_{j=0}((k+j)^{2}-(n-1)^{2}/4-\Delta_{\mathbb{H}})\right]^{-1}\\
=&\frac{1}{4\Gamma(l)\pi^{m+\frac{1}{2}}\left(\sinh \rho\right)^{2m-1}}
\sum^{m-k-l}_{j=0}\frac{\Gamma(m+k)\Gamma(m-k-l+1)\Gamma(m-j-l+\frac{1}{2})}{\Gamma(j+1)\Gamma(m-k-l-j+1)\Gamma(m+k-j)}
\left(\sinh\frac{\rho}{2}\right)^{2j+2l-2}.
\end{split}
\end{equation}
\end{lemma}
\begin{proof}
We  prove it by induction. By Lemma \ref{lm4.1},  (\ref{4.6}) is valid for
$l=1$. Assume (\ref{4.6}) is valid for $l>1$. Then for $l+1$, we have
\begin{equation}\label{4.7}
\begin{split}
&\left[\prod^{l}_{j=0}((k+j)^{2}-(n-1)^{2}/4-\Delta_{\mathbb{H}})\right]^{-1}\\
=&\frac{1}{(2k+l)l}\left\{\left[\prod^{l-1}_{j=0}((k+j)^{2}-(n-1)^{2}/4-\Delta_{\mathbb{H}})\right]^{-1}-
\left[\prod^{l}_{j=1}((k+j)^{2}-(n-1)^{2}/4-\Delta_{\mathbb{H}})\right]^{-1}\right\}\\
=&\frac{1}{(2k+l)l}\left\{\left[\prod^{l-1}_{j=0}((k+j)^{2}-(n-1)^{2}/4-\Delta_{\mathbb{H}})\right]^{-1}-
\left[\prod^{l-1}_{j=0}((k+1+j)^{2}-(n-1)^{2}/4-\Delta_{\mathbb{H}})\right]^{-1}\right\}
\\
=&\frac{\pi^{-m-\frac{1}{2}}\left(\sinh \rho\right)^{1-2m}}{(2k+l)l4\Gamma(l)}
\left[\sum^{m-k-l}_{j=0}\frac{\Gamma(m+k)\Gamma(m-k-l+1)\Gamma(m-j-l+\frac{1}{2})}{\Gamma(j+1)\Gamma(m-k-l-j+1)\Gamma(m+k-j)}
\left(\sinh\frac{\rho}{2}\right)^{2j+2l-2}-\right.\\
&\left.\sum^{m-k-1-l}_{j=0}\frac{\Gamma(m+k+1)\Gamma(m-k-l)\Gamma(m-j-l+\frac{1}{2})}{\Gamma(j+1)\Gamma(m-k-l-j)\Gamma(m+k+1-j)}
\left(\sinh\frac{\rho}{2}\right)^{2j+2l-2}\right].
\end{split}
\end{equation}
We compute
\begin{equation}\label{4.8}
\begin{split}
&\frac{\Gamma(m+k)\Gamma(m-k-l+1)}{\Gamma(m-k-l-j+1)\Gamma(m+k-j)}-\frac{\Gamma(m+k+1)\Gamma(m-k-l)}{\Gamma(m-k-l-j)\Gamma(m+k+1-j)}\\
=&\frac{\Gamma(m+k)\Gamma(m-k-l+1)}{\Gamma(m-k-l-j+1)\Gamma(m+k-j)}\left[1-\frac{(m+k)(m-k-l-j)}{(m+k-j)(m-k-l)}\right]\\
=&\frac{\Gamma(m+k)\Gamma(m-k-l+1)}{\Gamma(m-k-l-j+1)\Gamma(m+k-j)}\cdot\frac{(2k+l)j}{(m+k-j)(m-k-l)}\\
=&(2k+l)\frac{j\Gamma(m+k)\Gamma(m-k-l)}{\Gamma(m-k-l-j+1)\Gamma(m+k+1-j)}.
\end{split}
\end{equation}
Therefore, combining (\ref{4.7}) and (\ref{4.8}) yields
\begin{equation}\label{4.9}
\begin{split}
&\left[\prod^{l}_{j=0}((k+j)^{2}-(n-1)^{2}/4-\Delta_{\mathbb{H}})\right]^{-1}\\
=&\frac{\pi^{-m-\frac{1}{2}}\left(\sinh \rho\right)^{1-2m}}{(2k+l)l4\Gamma(l)}
\left[\frac{\Gamma(m+k)\Gamma(k+\frac{1}{2})}{\Gamma(2k+l)}\left(\sinh\frac{\rho}{2}\right)^{2m-2k-2}+\right.\\
&\left.(2k+l)\sum^{m-k-1-l}_{j=1}\frac{j\Gamma(m+k)\Gamma(m-k-l)\Gamma(m-j-l+\frac{1}{2})}{\Gamma(j+1)\Gamma(m-k-l-j+1)\Gamma(m+k+1-j)}
\left(\sinh\frac{\rho}{2}\right)^{2j+2l-2}\right].\\
=&\frac{\pi^{-m-\frac{1}{2}}\left(\sinh \rho\right)^{1-2m}}{4\Gamma(l+1)}
\left[\frac{\Gamma(m+k)\Gamma(k+\frac{1}{2})}{\Gamma(2k+l+1)}\left(\sinh\frac{\rho}{2}\right)^{2m-2k-2}+\right.\\
&\left.\sum^{m-k-2-l}_{j=0}\frac{\Gamma(m+k)\Gamma(m-k-l)\Gamma(m-j-l-\frac{1}{2})}{\Gamma(j+1)\Gamma(m-k-l-j)\Gamma(m+k-j)}
\left(\sinh\frac{\rho}{2}\right)^{2j+2l}\right].\\
=&\frac{\pi^{-m-\frac{1}{2}}\left(\sinh \rho\right)^{1-2m}}{4\Gamma(l+1)}\sum^{m-k-1-l}_{j=0}\frac{\Gamma(m+k)\Gamma(m-k-l)\Gamma(m-j-l-\frac{1}{2})}{\Gamma(j+1)\Gamma(m-k-l-j)\Gamma(m+k-j)}
\left(\sinh\frac{\rho}{2}\right)^{2j+2l}.
\end{split}
\end{equation}
The desired result follows.
\end{proof}

Before the proof of the next lemma, we recall the  Hardy-Littlewood-Sobolev inequality on hyperbolic space which was first proved by Beckner  \cite{be1} on upper half spaces (see also another proof in \cite{LuYang3} for the equivalent form on hyperbolic balls).

\begin{theorem}\label{th4.1}
Let $0<\lambda<n$ and $p=\frac{2n}{2n-\lambda}$. Then for $f,g\in L^{p}(\mathbb{B}^{n})$,
\begin{equation}\label{4.10}
\left|\int_{\mathbb{B}^{n}}\int_{\mathbb{B}^{n}}\frac{f(x)g(y)}{\left(2\sinh\frac{\rho(T_{y}(x))}{2}\right)^{\lambda}}dV_{x}dV_{y}\right|\leq C_{n,\lambda}\|f\|_{p}\|g\|_{p},
\end{equation}
 where
\begin{equation}\label{4.11}
C_{n,\lambda}=\pi^{\lambda/2}\frac{\Gamma(n/2-\lambda/2)}{\Gamma(n-\lambda/2)}\left(\frac{\Gamma(n/2)}{\Gamma(n)}\right)^{-1+\lambda/n}
\end{equation}
 is the best constant for the classical Hardy-Littlewood-Sobolev constant on $\mathbb{R}^{n}$.
Furthermore, the constant $C_{n,\lambda}$
 is sharp for the inequality (\ref{4.10}) and there is no nonzero extremal function for the inequality (\ref{4.10}).
\end{theorem}

We are now ready to prove the following

\begin{lemma}\label{lm4.4} Let $n\geq5$ be odd. There holds,  for each $u\in C^{\infty}_{0}(\mathbb{B}^{n})$,
\begin{equation*}
\int_{\mathbb{B}^{n}}\prod^{(n-3)/2}_{j=0}(j^{2}-(n-1)^{2}/4-\Delta_{\mathbb{H}}) u\cdot udV\geq S_{n,(n-1)/2}\left(\int_{\mathbb{B}^{n}}|u|^{2n}dV\right)^{\frac{1}{n}}.
\end{equation*}
\end{lemma}
\begin{proof}
It is sufficient to show the following inequality:
\begin{equation}\label{4.12}
\begin{split}
&\int_{\mathbb{B}^{n}}u(x) \left[\left(\prod^{(n-3)/2}_{j=0}(j^{2}-(n-1)^{2}/4-\Delta_{\mathbb{H}})\right)^{-1}u\right](x)dV\\
\leq& \frac{C_{n,1}}{\gamma_{n}(n-1)}\left(\int_{\mathbb{B}^{n}}|u(x)|^{\frac{2n}{2n-1}}dV\right)^{\frac{2n-1}{n}}.
\end{split}
\end{equation}

Choosing $k=0$ and $l=m$ in  Lemma \ref{lm4.2}, we have
\begin{equation}\label{4.14}
\begin{split}
\left[\prod^{(n-3)/2}_{j=0}(j^{2}-(n-1)^{2}/4-\Delta_{\mathbb{H}})\right]^{-1}
=&\frac{1}{4\Gamma(\frac{n-1}{2})\pi^{\frac{n}{2}}\left(\sinh \rho\right)^{n-2}}
\Gamma\left(\frac{1}{2}\right)\left(\sinh\frac{\rho}{2}\right)^{n-3}\\
=&\frac{\Gamma\left(\frac{1}{2}\right)}{4\Gamma(\frac{n-1}{2})\pi^{\frac{n}{2}}}\cdot\frac{1}{2^{n-2}(\cosh\frac{\rho}{2})^{n-2}\sinh\frac{\rho}{2}}\\
=&\frac{1}{\gamma_{n}(n-1)}\frac{1}{(\cosh\frac{\rho}{2})^{n-2}2\sinh\frac{\rho}{2}}\\
\leq&\frac{1}{\gamma_{n}(n-1)}\cdot\frac{1}{2\sinh\frac{\rho}{2}}.
\end{split}
\end{equation}
Therefore, by Theorem \ref{th4.1}, we have
\begin{equation*}
\begin{split}
&\int_{\mathbb{B}^{n}}u(x) \left[\left(\prod^{(n-1)/2}_{k=0}(k^{2}-(n-1)^{2}/4-\Delta_{\mathbb{H}})\right)^{-1}u\right](x)dV\\
\leq& \frac{1}{\gamma_{n}(n-1)}\int_{\mathbb{B}^{n}}\int_{\mathbb{B}^{n}}\frac{|u(x)|\cdot |u(y)|}{2\sinh\frac{\rho(T_{y}(x))}{2}}dV_{x}dV_{y}\\
\leq&\frac{C_{n,1}}{\gamma_{n}(n-1)}\left(\int_{\mathbb{B}^{n}}|u(x)|^{\frac{2n}{2n-1}}dV\right)^{\frac{2n-1}{n}}.
\end{split}
\end{equation*}
This proves inequality (\ref{4.12}). The proof of Lemma \ref{lm4.4} is thus completed.
\end{proof}

\textbf{Proof of Theorem \ref{th1.1}}. By Lemma \ref{lm4.4}, it is enough to show
\begin{equation}\label{4.15}
\begin{split}
\int_{\mathbb{B}^{n}}(P_{(n-1)/2}u)udV- \prod^{(n-1)/2}_{j=1}\frac{(2j-1)^{2}}{4}\int_{\mathbb{B}^{n}}u^{2}dV\geq
\int_{\mathbb{B}^{n}}\prod^{(n-3)/2}_{j=0}(j^{2}-(n-1)^{2}/4-\Delta_{\mathbb{H}}) u\cdot udV.
\end{split}
\end{equation}
By Plancherel's formula, we have
\begin{equation}\label{4.16}
\begin{split}
&\int_{\mathbb{B}^{n}}(P_{(n-1)/2}u)udV- \prod^{(n-1)/2}_{j=1}\frac{(2j-1)^{2}}{4}\int_{\mathbb{B}^{n}}u^{2}dV\\
=&D_{n}\int^{+\infty}_{-\infty}\int_{\mathbb{S}^{n-1}}\prod^{(n-1)/2}_{j=1}\left(\frac{(n-1)^{2}+\lambda^{2}}{4}-\frac{(n+2j-2)(n-2j)}{4}\right)
|\widehat{f}(\lambda,\zeta)|^{2}|\mathfrak{c}(\lambda)|^{-2}d\lambda d\sigma(\varsigma)\\
&-\prod^{(n-1)/2}_{j=1}\frac{(2j-1)^{2}}{4}
\cdot D_{n}\int^{+\infty}_{-\infty}\int_{\mathbb{S}^{n-1}}
|\widehat{f}(\lambda,\zeta)|^{2}|\mathfrak{c}(\lambda)|^{-2}d\lambda d\sigma(\varsigma)\\
=&D_{n}\int^{+\infty}_{-\infty}\int_{\mathbb{S}^{n-1}}\left(\prod^{(n-1)/2}_{j=1}\frac{(2j-1)^{2}+\lambda^{2}}{4}-\prod^{(n-1)/2}_{j=1}\frac{(2j-1)^{2}}{4}
\right)|\widehat{f}(\lambda,\zeta)|^{2}|\mathfrak{c}(\lambda)|^{-2}d\lambda d\sigma(\varsigma);
\end{split}
\end{equation}
and
\begin{equation}\label{4.17}
\begin{split}
&\int_{\mathbb{B}^{n}}\prod^{(n-3)/2}_{j=0}(j^{2}-(n-1)^{2}/4-\Delta_{\mathbb{H}}) u\cdot udV\\
=&D_{n}\int^{+\infty}_{-\infty}\int_{\mathbb{S}^{n-1}}\prod^{(n-3)/2}_{j=0}\frac{j^{2}+\lambda^{2}}{4}
|\widehat{f}(\lambda,\zeta)|^{2}|\mathfrak{c}(\lambda)|^{-2}d\lambda d\sigma(\varsigma).
\end{split}
\end{equation}
Since
\begin{equation}\label{4.18}
\begin{split}
\prod^{(n-1)/2}_{j=1}\frac{(2j-1)^{2}+\lambda^{2}}{4}-\prod^{(n-1)/2}_{j=1}\frac{(2j-1)^{2}}{4}
\geq&\frac{\lambda^{2}}{4}\prod^{(n-1)/2}_{j=2}\frac{(2j-1)^{2}+\lambda^{2}}{4}\\
=&\frac{\lambda^{2}}{4}\prod^{(n-3)/2}_{j=1}\frac{(2j+1)^{2}+\lambda^{2}}{4}\\
\geq&\frac{\lambda^{2}}{4}\prod^{(n-3)/2}_{j=1}\frac{j^{2}+\lambda^{2}}{4}\\
=&\prod^{(n-3)/2}_{j=0}\frac{j^{2}+\lambda^{2}}{4},
\end{split}
\end{equation}
we get, by (\ref{4.16}) and (\ref{4.17}),
\begin{equation}\label{4.19}
\begin{split}
&\int_{\mathbb{B}^{n}}(P_{(n-1)/2}u)udV- \prod^{(n-1)/2}_{j=1}\frac{(2j-1)^{2}}{4}\int_{\mathbb{B}^{n}}u^{2}dV\\
=&D_{n}\int^{+\infty}_{-\infty}\int_{\mathbb{S}^{n-1}}\left(\prod^{(n-1)/2}_{j=1}\frac{(2j-1)^{2}+\lambda^{2}}{4}-\prod^{(n-1)/2}_{j=1}\frac{(2j-1)^{2}}{4}
\right)|\widehat{f}(\lambda,\zeta)|^{2}|\mathfrak{c}(\lambda)|^{-2}d\lambda d\sigma(\varsigma)\\
\geq&D_{n}\int^{+\infty}_{-\infty}\int_{\mathbb{S}^{n-1}}\prod^{(n-3)/2}_{j=0}\frac{j^{2}+\lambda^{2}}{4}
|\widehat{f}(\lambda,\zeta)|^{2}|\mathfrak{c}(\lambda)|^{-2}d\lambda d\sigma(\varsigma)\\
=&\int_{\mathbb{B}^{n}}\prod^{(n-3)/2}_{j=0}(j^{2}-(n-1)^{2}/4-\Delta_{\mathbb{H}}) u\cdot udV.
\end{split}
\end{equation}
This proves the inequality (\ref{4.15}).  The proof of Theorem \ref{th1.1} is thereby completed.

\section{Estimate of a lower bound for the coefficient of the Hardy term in the Hardy-Sobolev-Maz'ya inequality: Proof of Theorem \ref{th1.2}}

In this section, we will establish a lower bound for the  coefficient $\lambda$ of the Hardy term in the higher order Hardy-Sobolev-Maz'ya inequality.

In what follows, $a=O(b)$ will stand for $a\leq C b$  and $a\sim b$ will stand for $C^{-1}b\leq a\leq C b$ with a positive constant $C$.

 Set
\begin{equation}\label{5.1}
\begin{split}
u_{k,\varepsilon}=&\left(\frac{1-|x|^{2}}{2}\right)^{\frac{n-2k}{2}}\left(\frac{2}{\varepsilon+|x|^{2}}\right)^{\frac{n-2k}{2}}.
\end{split}
\end{equation}
Let $0<\delta<1$. For $|x|<\delta$, we have, by the Taylor series of $\left(\frac{1}{1-t}\right)^{\frac{n-2k}{2}}$,
\begin{equation}\label{5.2}
\begin{split}
u_{k,\varepsilon}=&\left(\frac{1-|x|^{2}}{2}\right)^{\frac{n-2k}{2}}
\left(\frac{2}{\delta^{2}+\varepsilon}\right)^{\frac{n-2k}{2}}\left(\frac{1}{1-\frac{\delta^{2}-|x|^{2}}{\delta^{2}+\varepsilon}}\right)^{\frac{n-2k}{2}}\\
=&\left(\frac{1-|x|^{2}}{2}\right)^{\frac{n-2k}{2}}\left(\frac{2}{\delta^{2}+\varepsilon}\right)^{\frac{n-2k}{2}}
\sum^{\infty}_{j=0}\frac{\Gamma(j+\frac{n-2k}{2})}{\Gamma(j+1)\Gamma(\frac{n-2k}{2})}
\left(\frac{\delta^{2}-|x|^{2}}{\delta^{2}+\varepsilon}\right)^{j}.
\end{split}
\end{equation}

 Set
\begin{equation}\label{5.3}
\begin{split}
f_{k,\varepsilon}=\left\{
                    \begin{array}{ll}
                    u_{k,\varepsilon}-\left(\frac{1-|x|^{2}}{2}\right)^{\frac{n-2k}{2}}\left(\frac{2}{\delta^{2}+\varepsilon}\right)^{\frac{n-2k}{2}}
\sum\limits^{k-1}_{j=0}
\frac{\Gamma(j+\frac{n-2k}{2})}{\Gamma(j+1)\Gamma(\frac{n-2k}{2})}
\left(\frac{1-|x|^{2}}{\delta^{2}+\varepsilon}\right)^{j}  , & \hbox{$|x|<\delta$;} \\
                      0, & \hbox{$\delta\leq |x|<1$.}
                    \end{array}
                  \right.
\end{split}
\end{equation}
It is known that
\begin{equation}\label{5.4}
\begin{split}
(-\Delta)^{k}\left(\frac{2}{\varepsilon+|x|^{2}}\right)^{\frac{n-2k}{2}}=\varepsilon^{k}\frac{\Gamma(\frac{n+2k}{2})}{\Gamma(\frac{n-2k}{2})}
\left(\frac{2}{\varepsilon+|x|^{2}}\right)^{\frac{n+2k}{2}}=\varepsilon^{k}\frac{S_{n,k}}{\omega^{2k/n}_{n}}
\left(\frac{2}{\varepsilon+|x|^{2}}\right)^{\frac{n+2k}{2}},
\end{split}
\end{equation}
where $\omega_{n}$ is the surface area of $\mathbb{S}^{n}$ and $S_{n,k}=\frac{\Gamma(\frac{n+2k}{2})}{\Gamma(\frac{n-2k}{2})}\omega^{2k/n}_{n}$ is the best Sobolev
constant.
Using the  following identity  (see \cite{liu}, Theorem 2.3):
\begin{equation}\label{5.5}
\left(\frac{1-|x|^{2}}{2}\right)^{k+\frac{n}{2}}(-\Delta)^{k}\left[\left(\frac{1-|x|^{2}}{2}\right)^{k-\frac{n}{2}}f\right]=P_{k}f,\;\; f\in C^{\infty}_{0}(\mathbb{B}^{n}),\;\;k\in\mathbb{N},
\end{equation}
 we have, for $|x|<\delta$,
\begin{equation}\label{5.6}
\begin{split}
P_{k}u_{k,\varepsilon}=&\varepsilon^{k}\frac{\Gamma(\frac{n+2k}{2})}{\Gamma(\frac{n-2k}{2})}
u^{\frac{n+2k}{n-2k}}_{k,\varepsilon};\\
P_{k}f_{k,\varepsilon}=&P_{k}u_{k,\varepsilon}=\varepsilon^{k}\frac{\Gamma(\frac{n+2k}{2})}{\Gamma(\frac{n-2k}{2})}
u^{\frac{n+2k}{n-2k}}_{k,\varepsilon}.\\
\end{split}
\end{equation}

The proof of Theorem \ref{th1.2} is divided into two parts.
\subsection{Case I: $n\geq4k$}
\begin{lemma} \label{lm5.1}For $\varepsilon>0$ small enough, we have
\begin{equation}\label{b5.7}
\begin{split}
\int_{\mathbb{B}^{n}}(P_{k}f_{k,\varepsilon})f_{k,\varepsilon}dV\leq&\varepsilon^{k-n/2}
\frac{\Gamma(\frac{n+2k}{2})}{\Gamma(\frac{n-2k}{2})}\int_{\mathbb{R}^{n}}
\left(\frac{2}{1+|x|^{2}}\right)^{n}dx;\\
\int_{\mathbb{B}^{n}}|f_{k,\varepsilon}|^{\frac{2n}{n-2k}}dV\geq&\varepsilon^{-n/2}
\left[\int_{\mathbb{R}^{n}}
\left(\frac{2}{1+|x|^{2}}\right)^{n}dx+O(\varepsilon^{\frac{n-2k}{2}})\right];\\
\int_{\mathbb{B}^{n}}f_{k,\varepsilon}^{2}dV\thicksim&\left\{
       \begin{array}{ll}
         \varepsilon^{2k-\frac{n}{2}}, & \hbox{$n>4k$;} \\
    -\ln\varepsilon, & \hbox{$n=4k$.}
       \end{array}
     \right.
\end{split}
\end{equation}
\end{lemma}
\begin{proof} Set $B_{\delta}=\{x\in\mathbb{B}^{n}: |x|<\delta\}$. We have, by (\ref{5.3}) and (\ref{5.6}),
\begin{equation}\label{5.7}
\begin{split}
\int_{\mathbb{B}^{n}}(P_{k}f_{k,\varepsilon})f_{k,\varepsilon}dV=&\varepsilon^{k}\frac{\Gamma(\frac{n+2k}{2})}{\Gamma(\frac{n-2k}{2})}\int_{B_{\delta}}
u^{\frac{n+2k}{n-2k}}_{k,\varepsilon}f_{k,\varepsilon}dV\leq \varepsilon^{k}\frac{\Gamma(\frac{n+2k}{2})}{\Gamma(\frac{n-2k}{2})}\int_{B_{\delta}}
u^{\frac{2n}{n-2k}}_{k,\varepsilon}dV\\
=&\varepsilon^{k}
\frac{\Gamma(\frac{n+2k}{2})}{\Gamma(\frac{n-2k}{2})}
\int_{B_{\delta}}\left(\frac{2}{\varepsilon+|x|^{2}}\right)^{n}dx\\
=&\varepsilon^{k-n/2}
\frac{\Gamma(\frac{n+2k}{2})}{\Gamma(\frac{n-2k}{2})}\int_{B_{\delta/\varepsilon}}
\left(\frac{2}{1+|x|^{2}}\right)^{n}dx\\
\leq& \varepsilon^{k-n/2}
\frac{\Gamma(\frac{n+2k}{2})}{\Gamma(\frac{n-2k}{2})}\int_{\mathbb{R}^{n}}
\left(\frac{2}{1+|x|^{2}}\right)^{n}dx.
\end{split}
\end{equation}

Next we prove the second inequality. Using the  inequality
\[(s-t)^{\alpha}\geq s^{\alpha}-\alpha s^{\alpha-1}t,\;\;s\geq t\geq0,\;\;\alpha\geq1,\]
we have
\begin{equation*}
\begin{split}
&\int_{\mathbb{B}^{n}}|f_{k,\varepsilon}|^{\frac{2n}{n-2k}}dV\\
=&
\int_{B_{\delta}}
\left( u_{k,\varepsilon}-\left(\frac{1-|x|^{2}}{2}\right)^{\frac{n-2k}{2}}
\left(\frac{2}{\delta^{2}+\varepsilon}\right)^{\frac{n-2k}{2}}
\sum^{k-1}_{j=0}
\frac{\Gamma(j+\frac{n-2k}{2})}{\Gamma(j+1)\Gamma(\frac{n-2k}{2})}
\left(\frac{1-|x|^{2}}{\delta^{2}+\varepsilon}\right)^{j}\right)^{\frac{2n}{n-2k}}dV\\
\geq&\int_{B_{\delta}} u_{k,\varepsilon}^{\frac{2n}{n-2k}}dV-\frac{2n}{n-2k}\int_{B_{\delta}} u_{k,\varepsilon}^{\frac{n+2k}{n-2k}}\cdot\\
&\left(\frac{1-|x|^{2}}{2}\right)^{\frac{n-2k}{2}}
\left(\frac{2}{\delta^{2}+\varepsilon}\right)^{\frac{n-2k}{2}}
\sum^{k-1}_{j=0}
\frac{\Gamma(j+\frac{n-2k}{2})}{\Gamma(j+1)\Gamma(\frac{n-2k}{2})}
\left(\frac{1-|x|^{2}}{\delta^{2}+\varepsilon}\right)^{j}dV\\
=&\int_{B_{\delta}}\left(\frac{2}{\varepsilon+|x|^{2}}\right)^{n}dx-\frac{2n}{n-2k}\cdot\\
&\sum^{k-1}_{j=0}
\frac{\Gamma(j+\frac{n-2k}{2})}{\Gamma(j+1)\Gamma(\frac{n-2k}{2})}\int_{B_{\delta}}\left(\frac{2}{\varepsilon+|x|^{2}}\right)^{\frac{n+2k}{2}}
\left(\frac{2}{\delta^{2}+\varepsilon}\right)^{\frac{n-2k}{2}}
\left(\frac{1-|x|^{2}}{\delta^{2}+\varepsilon}\right)^{j}dx.
\end{split}
\end{equation*}
Easy computations lead to
\begin{equation*}
\begin{split}
\int_{B_{\delta}}\left(\frac{2}{\varepsilon+|x|^{2}}\right)^{n}=&\varepsilon^{-n/2}
\left[\int_{\mathbb{R}^{n}}
\left(\frac{2}{1+|x|^{2}}\right)^{n}dx+O(\varepsilon^{\frac{n}{2}})\right];\\
\int_{B_{\delta}}\left(\frac{2}{\varepsilon+|x|^{2}}\right)^{\frac{n+2k}{2}}&
\left(\frac{2}{\delta^{2}+\varepsilon}\right)^{\frac{n-2k}{2}}
\left(\frac{1-|x|^{2}}{\delta^{2}+\varepsilon}\right)^{j}dx=O(\varepsilon^{-k}),\;\;j=0,1,\cdots,k-1.
\end{split}
\end{equation*}
Therefore, we have
\begin{equation*}
\begin{split}
&\int_{\mathbb{B}^{n}}|f_{k,\varepsilon}|^{\frac{2n}{n-2k}}dV=\varepsilon^{-n/2}
\left[\int_{\mathbb{R}^{n}}
\left(\frac{2}{1+|x|^{2}}\right)^{n}dx+O(\varepsilon^{\frac{n-2k}{2}})\right]
\end{split}
\end{equation*}

Finally,  we prove the third inequality. We compute
\begin{equation}\label{5.9}
\begin{split}
\int_{\mathbb{B}^{n}}f_{k,\varepsilon}^{2}dV=&\int_{B_{\delta}}u_{k,\varepsilon}^{2}dV-2\int_{B_{\delta}}u_{k,\varepsilon}
\left(\frac{1-|x|^{2}}{2}\right)^{\frac{n-2k}{2}}\cdot\\
&\left(\frac{2}{\delta^{2}+\varepsilon}\right)^{\frac{n-2k}{2}}
\sum\limits^{k-1}_{j=0}
\frac{\Gamma(j+\frac{n-2k}{2})}{\Gamma(j+1)\Gamma(\frac{n-2k}{2})}
\left(\frac{1-|x|^{2}}{\delta^{2}+\varepsilon}\right)^{j}dV+\\
&\int_{B_{\delta}}\left|
\left(\frac{1-|x|^{2}}{2}\right)^{\frac{n-2k}{2}}\left(\frac{2}{\delta^{2}+\varepsilon}\right)^{\frac{n-2k}{2}}
\sum\limits^{k-1}_{j=0}
\frac{\Gamma(j+\frac{n-2k}{2})}{\Gamma(j+1)\Gamma(\frac{n-2k}{2})}
\left(\frac{1-|x|^{2}}{\delta^{2}+\varepsilon}\right)^{j}\right|^{2}dV.
\end{split}
\end{equation}
Simple calculations lead to
\begin{equation*}
\begin{split}
\int_{\mathbb{B}^{n}}u_{k,\varepsilon}^{2}dV=&
\int_{B_{\delta}}\left(\frac{2}{\varepsilon+|x|^{2}}\right)^{n-2k}\left(\frac{2}{1-|x|^{2}}\right)^{2k}dx\\
\thicksim&\left\{
       \begin{array}{ll}
         \varepsilon^{2k-\frac{n}{2}}, & \hbox{$n>4k$;} \\
    -\ln\varepsilon, & \hbox{$n=4k$,}
       \end{array}
     \right.\\
\int_{B_{\delta}}u_{k,\varepsilon}
\left(\frac{1-|x|^{2}}{2}\right)^{\frac{n-2k}{2}}
&\left(\frac{2}{\delta^{2}+\varepsilon}\right)^{\frac{n-2k}{2}}
\left(\frac{1-|x|^{2}}{\delta^{2}+\varepsilon}\right)^{j}dV=O(\varepsilon^{k}),\;\;k=0,1,\cdots,k-1,\\
\int_{B_{\delta}}\left|
\left(\frac{1-|x|^{2}}{2}\right)^{\frac{n-2k}{2}}\right.&\left.\left(\frac{2}{\delta^{2}+\varepsilon}\right)^{\frac{n-2k}{2}}
\sum\limits^{k-1}_{j=0}
\frac{\Gamma(j+\frac{n-2k}{2})}{\Gamma(j+1)\Gamma(\frac{n-2k}{2})}
\left(\frac{1-|x|^{2}}{\delta^{2}+\varepsilon}\right)^{j}\right|^{2}dV\thicksim1.
\end{split}
\end{equation*}
Therefore, we have, by (\ref{5.9}),
\begin{equation*}
\begin{split}
\int_{\mathbb{B}^{n}}f_{k,\varepsilon}^{2}dV
\thicksim&\left\{
       \begin{array}{ll}
         \varepsilon^{2k-\frac{n}{2}}, & \hbox{$n>4k$;} \\
    -\ln\varepsilon, & \hbox{$n=4k$.}
       \end{array}
     \right.
\end{split}
\end{equation*}
\end{proof}

\begin{lemma} \label{lm5.2}
Let $2\leq k\leq\frac{n}{4}$.  Suppose that there exists $\lambda\in\mathbb{R}$ such that for
any $u\in C^{\infty}_{0}(\mathbb{B}^{n})$,
\begin{equation}\label{b5.10}
\begin{split}
\int_{\mathbb{B}^{n}}(P_{k}u)udV+\lambda \int_{\mathbb{B}^{n}}u^{2}dV\geq S_{n,k}
\left(\int_{\mathbb{B}^{n}}|u|^{\frac{2n}{n-2k}}dV\right)^{\frac{n-2k}{n}},
\end{split}
\end{equation}
where $S_{n,k}$ is the best $k$-th order Sobolev constant. Then $\lambda\geq0$.
\end{lemma}
\begin{proof}
 By Lemma \ref{lm5.1}, if $n>4k$, then for $\varepsilon>0$ small enough, we have, for some $C_{1}>0$,
\begin{equation}\label{5.10}
\begin{split}
&\frac{\int_{\mathbb{B}^{n}}P_{k}f_{k,\varepsilon} f_{k,\varepsilon}dV+\lambda\int_{\mathbb{B}^{n}}f^{2}_{k,\varepsilon} dV }{\left(\int_{\mathbb{B}^{n}}|f_{k,\varepsilon}|^{\frac{2n}{n-2k}}dV\right)^{\frac{n-2k}{n}}}
\leq
\frac{\frac{\Gamma(\frac{n+2k}{2})}{\Gamma(\frac{n-2k}{2})}\int_{\mathbb{R}^{n}}
\left(\frac{2}{1+|x|^{2}}\right)^{n}dx+C_{1}\lambda
\varepsilon^{k}}{
\left[\int_{\mathbb{R}^{n}}
\left(\frac{2}{1+|x|^{2}}\right)^{n}dx+O(\varepsilon^{\frac{n-2k}{2}})\right]^{\frac{n-2k}{n}}}\\
=&\frac{\Gamma(\frac{n+2k}{2})}{\Gamma(\frac{n-2k}{2})}
\left[\int_{\mathbb{R}^{n}}
\left(\frac{2}{1+|x|^{2}}\right)^{n}dx\right]^{\frac{2k}{n}}+
O(\varepsilon^{\frac{n-2k}{2}})+\frac{C_{1}\lambda\varepsilon^{k}}{
\left[\int_{\mathbb{R}^{n}}
\left(\frac{2}{1+|x|^{2}}\right)^{n}dx\right]^{\frac{n-2k}{n}}}+O(\varepsilon^{\frac{n}{2}})\\
=&S_{n,k}+\frac{C_{1}\lambda\varepsilon^{k}}{
\left[\int_{\mathbb{R}^{n}}
\left(\frac{2}{1+|x|^{2}}\right)^{n}dx\right]^{\frac{n-2k}{n}}}+O(\varepsilon^{\frac{n-2k}{2}}).
\end{split}
\end{equation}
To get the last equality, we use the fact
$$
\int_{\mathbb{R}^{n}}
\left(\frac{2}{1+|x|^{2}}\right)^{n}dx=2^{n}\omega_{n-1}\int^{\infty}_{0}\frac{r^{n-1}}{(1+r^{2})^{n}}dr=2^{n-1}\omega_{n-1}\int^{\infty}_{0}
\frac{t^{\frac{n}{2}-1}}{(1+t)^{n}}dt=\omega_{n}.
$$
Since (\ref{b5.10}) implies
\begin{equation*}
\begin{split}
&\frac{\int_{\mathbb{B}^{n}}P_{k}f_{k,\varepsilon} f_{k,\varepsilon}dV+\lambda\int_{\mathbb{B}^{n}}f^{2}_{k,\varepsilon} dV }{\left(\int_{\mathbb{B}^{n}}|f_{k,\varepsilon}|^{\frac{2n}{n-2k}}dV\right)^{\frac{n-2k}{n}}}
\geq S_{n,k},
\end{split}
\end{equation*}
we  have, by (\ref{5.10}),  $\lambda\geq0$.

Similarly, in the case $n=4k$, we have, for some $C_{2}>0$,
\begin{equation*}
\begin{split}
S_{n,k}\leq&\frac{\int_{\mathbb{B}^{n}}P_{k}f_{k,\varepsilon} f_{k,\varepsilon}dV+\lambda\int_{\mathbb{B}^{n}}f^{2}_{k,\varepsilon} dV }{\left(\int_{\mathbb{B}^{n}}|f_{k,\varepsilon}|^{\frac{2n}{n-2k}}dV\right)^{\frac{n-2k}{n}}}
\leq
\frac{\frac{\Gamma(\frac{n+2k}{2})}{\Gamma(\frac{n-2k}{2})}\int_{\mathbb{R}^{n}}
\left(\frac{2}{1+|x|^{2}}\right)^{n}dx-C_{2}\lambda
\varepsilon^{k}\ln\varepsilon}{
\left[\int_{\mathbb{R}^{n}}
\left(\frac{2}{1+|x|^{2}}\right)^{n}dx+O(\varepsilon^{k})\right]^{\frac{n-2k}{n}}}
\\
\leq&\frac{\Gamma(\frac{n+2k}{2})}{\Gamma(\frac{n-2k}{2})}
\left[\int_{\mathbb{R}^{n}}
\left(\frac{2}{1+|x|^{2}}\right)^{n}dx\right]^{\frac{2k}{n}}+
O(\varepsilon^{k})-\frac{C_{2}\lambda\varepsilon^{k}\ln\varepsilon}{
\left[\int_{\mathbb{R}^{n}}
\left(\frac{2}{1+|x|^{2}}\right)^{n}dx\right]^{\frac{n-2k}{n}}}+O(\varepsilon^{2k}\ln\varepsilon)\\
=&S_{n,k}-\frac{C_{2}\lambda\varepsilon^{k}\ln\varepsilon}{
\left[\int_{\mathbb{R}^{n}}
\left(\frac{2}{1+|x|^{2}}\right)^{n}dx\right]^{\frac{n-2k}{n}}}+O(\varepsilon^{k}).
\end{split}
\end{equation*}
Therefore, we must also have $\lambda\geq 0$.
\end{proof}

\subsection{Case II: $2k+2\leq n<4k$} We choose $\delta=1$ in (\ref{5.2}) and set
\begin{equation}\label{5.12}
\begin{split}
g_{k,\varepsilon}=u_{k,\varepsilon}-\left(\frac{1-|x|^{2}}{2}\right)^{\frac{n-2k}{2}}\left(\frac{2}{1+\varepsilon}\right)^{\frac{n-2k}{2}}\sum^{k-1}_{j=0}
\frac{\Gamma(j+\frac{n-2k}{2})}{\Gamma(j+1)\Gamma(\frac{n-2k}{2})}
\left(\frac{1-|x|^{2}}{1+\varepsilon}\right)^{j}.
\end{split}
\end{equation}
Then
\begin{equation}\label{5.13}
\begin{split}
P_{k}g_{k,\varepsilon}=&P_{k}u_{k,\varepsilon}=\varepsilon^{k}\frac{\Gamma(\frac{n+2k}{2})}{\Gamma(\frac{n-2k}{2})}
u^{\frac{n+2k}{n-2k}}_{k,\varepsilon}
\end{split}
\end{equation}
and
\begin{equation}\label{5.15}
\begin{split}
\int_{\mathbb{B}^{n}}(P_{k}g_{k,\varepsilon})g_{k,\varepsilon}dV\leq& \varepsilon^{k}\frac{\Gamma(\frac{n+2k}{2})}{\Gamma(\frac{n-2k}{2})}\int_{B_{\delta}}
u^{\frac{2n}{n-2k}}_{k,\varepsilon}dV\\
=&\varepsilon^{k}
\frac{\Gamma(\frac{n+2k}{2})}{\Gamma(\frac{n-2k}{2})}
\int_{B_{\delta}}\left(\frac{2}{\varepsilon+|x|^{2}}\right)^{n}dx\\
\leq& \varepsilon^{k-n/2}
\frac{\Gamma(\frac{n+2k}{2})}{\Gamma(\frac{n-2k}{2})}\int_{\mathbb{R}^{n}}
\left(\frac{2}{1+|x|^{2}}\right)^{n}dx.
\end{split}
\end{equation}
Easy calculations lead to
\begin{equation}\label{5.16}
\begin{split}
&\int_{\mathbb{B}^{n}}|g_{k,\varepsilon})|^{2}dV\\
=&\int_{\mathbb{B}^{n}}\left|\left(\frac{2}{\varepsilon+|x|^{2}}\right)^{\frac{n-2k}{2}}-\left(\frac{2}{1+\varepsilon}\right)^{\frac{n-2k}{2}}\sum^{k-1}_{j=0}
\frac{\Gamma(j+\frac{n-2k}{2})}{\Gamma(j+1)\Gamma(\frac{n-2k}{2})}
\left(\frac{1-|x|^{2}}{1+\varepsilon}\right)^{j}\right|^{2}\frac{2^{2k}}{(1-|x|^{2})^{2k}}dx\\
=&2^{n}\int_{\mathbb{B}^{n}}\left|\frac{1}{|x|^{n-2k}}-\sum^{k-1}_{j=0}
\frac{\Gamma(j+\frac{n-2k}{2})}{\Gamma(j+1)\Gamma(\frac{n-2k}{2})}
\left(1-|x|^{2}\right)^{j}\right|^{2}\frac{1}{(1-|x|^{2})^{2k}}dx+o(1).
\end{split}
\end{equation}
\begin{lemma} \label{lm5.3}For $\varepsilon>0$ small enough, we have
\begin{equation*}
\begin{split}
\left(\int_{\mathbb{B}^{n}}|g_{k,\varepsilon}|^{\frac{2n}{n-2k}}dx\right)^{\frac{n-2k}{n}}
&\geq\varepsilon^{k-\frac{n}{2}}\left(\int_{\mathbb{R}^{n}}\frac{1}{(1+|x|^{2})^{n}}dx\right)^{\frac{n-2k}{n}}\cdot\\
&\left[1-
\varepsilon^{\frac{n-2k}{2}}2^{\frac{n}{2}-k+1}\sum^{k-1}_{j=0}
\frac{\Gamma(j+\frac{n-2k}{2})}{\Gamma(j+1)\Gamma(\frac{n-2k}{2})}\frac{\int_{\mathbb{R}^{n}}\frac{1}{(1+|x|^{2})^{\frac{n+2k}{2}}}
dx}{\int_{\mathbb{R}^{n}}\frac{1}{(1+|x|^{2})^{n}}dx}+O(\varepsilon^{\frac{n}{2}})\right].
\end{split}
\end{equation*}
\end{lemma}
\begin{proof}
We have, for $\varepsilon>0$ small enough,
\begin{equation}\label{5.17}
\begin{split}
&\int_{\mathbb{B}^{n}}|g_{k,\varepsilon}|^{\frac{2n}{n-2k}}dx\\
=&\int_{\mathbb{B}^{n}}\left|u_{k,\varepsilon}
-\left(\frac{2}{1+\varepsilon}\right)^{\frac{n-2k}{2}}\sum^{k-1}_{j=0}
\frac{\Gamma(j+\frac{n-2k}{2})}{\Gamma(j+1)\Gamma(\frac{n-2k}{2})}
\left(\frac{1-|x|^{2}}{1+\varepsilon}\right)^{j}\right|^{\frac{2n}{n-2k}}dx\\
=&\int_{\mathbb{B}^{n}}\left|\left(\frac{1}{\varepsilon+|x|^{2}}\right)^{\frac{n-2k}{2}}-\left(\frac{2}
{1+\varepsilon}\right)^{\frac{n-2k}{2}}\sum^{k-1}_{j=0}
\frac{\Gamma(j+\frac{n-2k}{2})}{\Gamma(j+1)\Gamma(\frac{n-2k}{2})}
\left(\frac{1-|x|^{2}}{1+\varepsilon}\right)^{j}\right|^{\frac{2n}{n-2k}}dx\\
\geq&\int_{\mathbb{B}^{n}}\left(\frac{1}{\varepsilon+|x|^{2}}\right)^{n}dx-\\
&\frac{2n}{n-2k}\left(\frac{2}{1+\varepsilon}\right)^{\frac{n-2k}{2}}
\sum^{k-1}_{j=0}
\frac{\Gamma(j+\frac{n-2k}{2})}{\Gamma(j+1)\Gamma(\frac{n-2k}{2})}\int_{\mathbb{B}^{n}}\left(\frac{1}{\varepsilon+|x|^{2}}\right)^{\frac{n+2k}{2}}
\left(\frac{1-|x|^{2}}{1+\varepsilon}\right)^{j}dx\\
\geq&\int_{\mathbb{B}^{n}}\left(\frac{1}{\varepsilon+|x|^{2}}\right)^{n}dx-
\frac{2^{\frac{n}{2}-k+1}n}{n-2k}
\sum^{k-1}_{j=0}
\frac{\Gamma(j+\frac{n-2k}{2})}{\Gamma(j+1)\Gamma(\frac{n-2k}{2})}\int_{\mathbb{B}^{n}}\left(\frac{1}{\varepsilon+|x|^{2}}\right)^{\frac{n+2k}{2}}
dx.
\end{split}
\end{equation}
An easy computation leads to
\begin{equation*}
\begin{split}
\int_{\mathbb{B}^{n}}\left(\frac{1}{\varepsilon+|x|^{2}}\right)^{n}dx=&\varepsilon^{-\frac{n}{2}}\left[\int_{\mathbb{R}^{n}}
\left(\frac{2}{1+|x|^{2}}\right)^{n}dx+O(\varepsilon^{\frac{n}{2}})\right];\\
\int_{\mathbb{B}^{n}}\left(\frac{1}{\varepsilon+|x|^{2}}\right)^{\frac{n+2k}{2}}
dx=&\varepsilon^{-k}\left[
\int_{\mathbb{R}^{n}}\left(\frac{1}{1+|x|^{2}}\right)^{\frac{n+2k}{2}}dx+O(\varepsilon^{k})\right]
dx.
\end{split}
\end{equation*}
Therefore, we have, by (\ref{5.17}),
\begin{equation*}
\begin{split}\int_{\mathbb{B}^{n}}|g_{k,\varepsilon}|^{\frac{2n}{n-2k}}&dx
\geq\varepsilon^{-\frac{n}{2}}\int_{\mathbb{R}^{n}}\frac{1}{(1+|x|^{2})^{n}}dx\cdot\\
&\left[1-
\frac{\varepsilon^{\frac{n-2k}{2}}2^{\frac{n}{2}-k+1}n}{n-2k}\sum^{k-1}_{j=0}
\frac{\Gamma(j+\frac{n-2k}{2})}{\Gamma(j+1)\Gamma(\frac{n-2k}{2})}\frac{\int_{\mathbb{R}^{n}}\frac{1}{(1+|x|^{2})^{\frac{n+2k}{2}}}
dx}{\int_{\mathbb{R}^{n}}\frac{1}{(1+|x|^{2})^{n}}dx}+O(\varepsilon^{\frac{n}{2}})\right].
\end{split}
\end{equation*}
The desired result follows.
\end{proof}

\begin{lemma} \label{lm5.4}
Let $k\geq2$ and  $2k+2\leq n\leq 4k-1$.  Suppose that there exists $\lambda\in\mathbb{R}$ such that for
any $u\in C^{\infty}_{0}(\mathbb{B}^{n})$,
\begin{equation*}
\begin{split}
\int_{\mathbb{B}^{n}}(P_{k}u)udV+\lambda \int_{\mathbb{B}^{n}}u^{2}dV\geq S_{n,k}
\left(\int_{\mathbb{B}^{n}}|u|^{\frac{2n}{n-2k}}dV\right)^{\frac{n-2k}{n}},
\end{split}
\end{equation*}
where $S_{n,k}$ is the best $k$-th order Sobolev constant. Then
\begin{equation*}
\begin{split}
\lambda
\geq-\frac{\Gamma(n/2)\Gamma(k)}{2^{\frac{n+2k}{2}}\Gamma(\frac{n-2k}{2})\int^{1}_{0}[r^{2k-n}-\sum\limits^{k-1}_{j=0}
\frac{\Gamma(j+\frac{n-2k}{2})}{\Gamma(j+1)\Gamma(\frac{n-2k}{2})}
\left(1-r^{2}\right)^{j}]^{2}\frac{r^{n-1}}{(1-r^{2})^{2k}}dr}
\sum^{k-1}_{j=0}
\frac{\Gamma(j+\frac{n-2k}{2})}{\Gamma(j+1)\Gamma(\frac{n-2k}{2})}.
\end{split}
\end{equation*}
\end{lemma}
\begin{proof}We have, by Lemma \ref{lm5.3},
\begin{equation*}
\begin{split}
S_{n,k}\leq&\frac{\int_{\mathbb{B}^{n}}P_{k}g_{k,\varepsilon} g_{k,\varepsilon}dV+\lambda\int_{\mathbb{B}^{n}}g^{2}_{k,\varepsilon} dV }{\left(\int_{\mathbb{B}^{n}}|g_{k,\varepsilon}|^{\frac{2n}{n-2k}}dV\right)^{\frac{n-2k}{n}}}
\\
\leq&S_{n,k}\frac{1+\varepsilon^{\frac{n}{2}-k}\lambda\frac{\Gamma(\frac{n-2k}{2})}{\Gamma(\frac{n+2k}{2})}\frac{2^{n}
\int_{\mathbb{B}^{n}}\left|\frac{1}{|x|^{n-2k}}-\sum\limits^{k-1}_{j=0}
\frac{\Gamma(j+\frac{n-2k}{2})}{\Gamma(j+1)\Gamma(\frac{n-2k}{2})}
\left(1-|x|^{2}\right)^{j}\right|^{2}\frac{1}{(1-|x|^{2})^{2k}}dx}{\int_{\mathbb{R}^{n}}\frac{1}{(1+|x|^{2})^{n}}dx}+o(\varepsilon^{\frac{n}{2}-k})}{
1-
\varepsilon^{\frac{n}{2}-k}2^{\frac{n}{2}-k+1}\sum\limits^{k-1}_{j=0}
\frac{\Gamma(j+\frac{n-2k}{2})}{\Gamma(j+1)\Gamma(\frac{n-2k}{2})}\frac{\int_{\mathbb{R}^{n}}\frac{1}{(1+|x|^{2})^{\frac{n+2k}{2}}}
dx}{\int_{\mathbb{R}^{n}}\frac{1}{(1+|x|^{2})^{n}}dx}+O(\varepsilon^{\frac{n}{2}})}.
\end{split}
\end{equation*}
Therefore, we must have
\begin{equation*}
\begin{split}
\lambda\geq&-\frac{2^{\frac{n}{2}-k+1}\frac{\Gamma(\frac{n+2k}{2})}{\Gamma(\frac{n-2k}{2})}\sum\limits^{k-1}_{j=0}
\frac{\Gamma(j+\frac{n-2k}{2})}{\Gamma(j+1)\Gamma(\frac{n-2k}{2})}\int_{\mathbb{R}^{n}}\frac{1}{(1+|x|^{2})^{\frac{n+2k}{2}}}
dx}{2^{n}
\int_{\mathbb{B}^{n}}\left|\frac{1}{|x|^{n-2k}}-\sum\limits^{k-1}_{j=0}
\frac{\Gamma(j+\frac{n-2k}{2})}{\Gamma(j+1)\Gamma(\frac{n-2k}{2})}
\left(1-|x|^{2}\right)^{j}\right|^{2}\frac{1}{(1-|x|^{2})^{2k}}dx}\\
=&-\frac{\frac{\Gamma(\frac{n+2k}{2})}{\Gamma(\frac{n-2k}{2})}\sum\limits^{k-1}_{j=0}
\frac{\Gamma(j+\frac{n-2k}{2})}{\Gamma(j+1)\Gamma(\frac{n-2k}{2})}\int^{\infty}_{0}\frac{r^{n-1}}{(1+r^{2})^{\frac{n+2k}{2}}}
dr}{2^{\frac{n}{2}+k-1}
\int^{\infty}_{0}\left|\frac{1}{r^{n-2k}}-\sum\limits^{k-1}_{j=0}
\frac{\Gamma(j+\frac{n-2k}{2})}{\Gamma(j+1)\Gamma(\frac{n-2k}{2})}
\left(1-r^{2}\right)^{j}\right|^{2}\frac{r^{n-1}dr}{(1-r^{2})^{2k}}}\\
=&-\frac{\Gamma(n/2)\Gamma(k)}{2^{\frac{n+2k}{2}}\Gamma(\frac{n-2k}{2})\int^{1}_{0}[r^{2k-n}-\sum\limits^{k-1}_{j=0}
\frac{\Gamma(j+\frac{n-2k}{2})}{\Gamma(j+1)\Gamma(\frac{n-2k}{2})}
\left(1-r^{2}\right)^{j}]^{2}\frac{r^{n-1}}{(1-r^{2})^{2k}}dr}
\sum^{k-1}_{j=0}
\frac{\Gamma(j+\frac{n-2k}{2})}{\Gamma(j+1)\Gamma(\frac{n-2k}{2})}.
\end{split}
\end{equation*}
To get the last equality, we use the fact
\begin{equation*}
\begin{split}
\int^{\infty}_{0}\frac{r^{n-1}}{(1+r^{2})^{\frac{n+2k}{2}}}dr=
\frac{1}{2}\int^{\infty}_{0}\frac{t^{\frac{n}{2}-1}}{(1+t)^{\frac{n+2k}{2}}}dt=\frac{1}{2}\mathbf{B}\left(\frac{n}{2},k\right)=
\frac{1}{2}\frac{\Gamma(n/2)\Gamma(k)}{\Gamma(k+n/2)}.
\end{split}
\end{equation*}
This completes the proof of Lemma \ref{lm5.4}.
\end{proof}

\textbf{Proof of Theorem \ref{th1.2}} Combining Lemma \ref{lm5.2} and \ref{lm5.4}, Theorem \ref{th1.2} follows.

\section{Precise expression and optimal bound for Green's function of Paneitz and GJMS operators: Proof of Theorems  \ref{th1.3}}

In this section, we will provide a proof of Theorems  \ref{th1.3}.
Namely,  we will provide a
precise expression and optimal bound for Green's function of Paneitz and GJMS operators $P_k$.

The proofs are divided into two parts.
\subsection{Case 1: $n$ is even}
\begin{lemma}\label{lm6.1}If $n=2m$ and $0\leq k\leq m-1$, then
\begin{equation}\label{6.1}
\begin{split}
&((k+1/2)^{2}-(n-1)^{2}/4-\Delta_{\mathbb{H}})^{-1}\\
=&\frac{1}{4\pi^{m+\frac{1}{2}}\left(\sinh \rho\right)^{2m-1}}\sum^{m-k-1}_{j=0}\frac{\Gamma(m+k)\Gamma(m-k)\Gamma(m-j-\frac{1}{2})}{\Gamma(j+1)\Gamma(m-k-j)\Gamma(m+k-j)}
\left(\sinh\frac{\rho}{2}\right)^{2j}.
\end{split}
\end{equation}
\end{lemma}
\begin{proof}
If $n=2m$ and $0\leq k\leq m-2$, then by (\ref{b2.9}) and the binomial theorem, we have
\begin{equation}\label{6.2}
\begin{split}
&((k+1/2)^{2}-(n-1)^{2}/4-\Delta_{\mathbb{H}})^{-1}\\
=&\frac{(2\pi)^{-\frac{n}{2}}\Gamma(m+k)}{2^{k+1}
\Gamma(k+1)}(\sinh\rho)^{2-2m}\int^{\pi}_{0}(\cosh\rho+\cos t)^{m-2-k}
(\sin t)^{2k+1}dt\\
=&\frac{(2\pi)^{-\frac{n}{2}}\Gamma(m+k)}{2^{k+1}
\Gamma(k+1)}(\sinh\rho)^{2-2m}\int^{\pi}_{0}(2\sinh^{2}\frac{\rho}{2}+1+\cos t)^{m-2-k}
(\sin t)^{2k+1}dt\\
=&\frac{(2\pi)^{-\frac{n}{2}}\Gamma(m+k)}{2^{k+1}
\Gamma(k+1)}(\sinh\rho)^{2-2m}\sum^{m-2-k}_{j=0}C^{j}_{m-2-k}(2\sinh^{2}\frac{\rho}{2})^{j}
\int^{\pi}_{0}(1+\cos t)^{m-2-k-j}
(\sin t)^{2k+1}dt.
\end{split}
\end{equation}

We have, by (\ref{6.2}) and (\ref{4.4}),
\begin{equation}\label{6.5}
\begin{split}
&((k+1/2)^{2}-(n-1)^{2}/4-\Delta_{\mathbb{H}})^{-1}\\
=&\frac{(2\pi)^{-\frac{n}{2}}\Gamma(m+k)}{2^{k+1}
\Gamma(k+1)}(\sinh\rho)^{2-2m}\cdot\\
&\sum^{m-2-k}_{j=0}\frac{\Gamma(m-1-k)}{\Gamma(j+1)\Gamma(m-1-k-j)}(2\sinh^{2}\frac{\rho}{2})^{j}
2^{m+k-1-j}\frac{\Gamma(m-j-1)\Gamma(k+1)}{\Gamma(m+k-j)}\\
=&\frac{1}{4\pi^{m}\left(\sinh \rho\right)^{2m-2}}\sum^{m-k-2}_{j=0}\frac{\Gamma(m+k)\Gamma(m-k-1)\Gamma(m-j-1)}{\Gamma(j+1)\Gamma(m-1-k-j)\Gamma(m+k-j)}
\left(\sinh\frac{\rho}{2}\right)^{2j}.
\end{split}
\end{equation}
\end{proof}

In general, we have the following:
\begin{lemma}\label{lm6.2}Let $n=2m$ and $l\in\{1,2,\cdots,m-1\}$. We have, for $k\in\{0,1,\cdots, m-1-l\}$ and $\rho>0$,
\begin{equation}\label{6.6}
\begin{split}
&\left[\prod^{l-1}_{j=0}((k+j+1/2)^{2}-(n-1)^{2}/4-\Delta_{\mathbb{H}})\right]^{-1}\\
=&\frac{1}{4\Gamma(l)\pi^{m}\left(\sinh \rho\right)^{2m-2}}
\sum^{m-1-k-l}_{j=0}\frac{\Gamma(m+k)\Gamma(m-k-l)\Gamma(m-j-l)}{\Gamma(j+1)\Gamma(m-k-l-j)\Gamma(m+k-j)}
\left(\sinh\frac{\rho}{2}\right)^{2j+2l-2}
\end{split}
\end{equation}
\end{lemma}
\begin{proof}
We  prove it by induction. By Lemma \ref{lm6.1},  (\ref{6.6}) is valid for
$l=1$. Assume (\ref{6.6}) is valid for $l>1$. Then for $l+1$, we have
\begin{equation}\label{6.7}
\begin{split}
&\left[\prod^{l}_{j=0}((k+j+1/2)^{2}-(n-1)^{2}/4-\Delta_{\mathbb{H}})\right]^{-1}\\
=&\frac{1}{(2k+1+l)l}\left\{\left[\prod^{l-1}_{j=0}((k+j+1/2)^{2}-(n-1)^{2}/4-\Delta_{\mathbb{H}})\right]^{-1}-
\right.
\\
&\left.
\left[\prod^{l}_{j=1}((k+j+1/2)^{2}-(n-1)^{2}/4-\Delta_{\mathbb{H}})\right]^{-1}\right\}\\
=&\frac{1}{(2k+1+l)l}\left\{\left[\prod^{l-1}_{j=0}((k+j+1/2)^{2}-(n-1)^{2}/4-\Delta_{\mathbb{H}})\right]^{-1}-
\right.
\\
&\left.
\left[\prod^{l-1}_{j=0}((k+j+1+1/2)^{2}-(n-1)^{2}/4-\Delta_{\mathbb{H}})\right]^{-1}\right\}\\
=&\frac{\pi^{-m}\left(\sinh \rho\right)^{2-2m}}{(2k+1+l)l4\Gamma(l)}
\left[\sum^{m-1-k-l}_{j=0}\frac{\Gamma(m+k)\Gamma(m-k-l)\Gamma(m-j-l)}{\Gamma(j+1)\Gamma(m-k-l-j)\Gamma(m+k-j)}
\left(\sinh\frac{\rho}{2}\right)^{2j+2l-2}\right.\\
&-\left.\sum^{m-2-k-l}_{j=0}\frac{\Gamma(m+k+1)\Gamma(m-1-k-l)\Gamma(m-j-l)}{\Gamma(j+1)\Gamma(m-1-k-l-j)\Gamma(m+k+1-j)}
\left(\sinh\frac{\rho}{2}\right)^{2j+2l-2}\right].
\end{split}
\end{equation}
We compute
\begin{equation}\label{6.8}
\begin{split}
&\frac{\Gamma(m+k)\Gamma(m-k-l)}{\Gamma(m-k-l-j)\Gamma(m+k-j)}-\frac{\Gamma(m+k+1)\Gamma(m-1-k-l)}{\Gamma(m-1-k-l-j)\Gamma(m+k+1-j)}\\
=&\frac{\Gamma(m+k)\Gamma(m-k-l)}{\Gamma(m-k-l-j)\Gamma(m+k-j)}\left[1-\frac{(m+k)(m-1-k-l-j)}{(m+k-j)(m-1-k-l)}\right]\\
=&\frac{\Gamma(m+k)\Gamma(m-k-l)}{\Gamma(m-k-l-j)\Gamma(m+k-j)}\cdot\frac{(2k+1+l)j}{(m+k-j)(m-1-k-l)}\\
=&(2k+1+l)\frac{j\Gamma(m+k)\Gamma(m-1-k-l)}{\Gamma(m-k-l-j)\Gamma(m+k+1-j)}.
\end{split}
\end{equation}
Therefore, combining (\ref{6.7}) and (\ref{6.8}) yields
\begin{equation}\label{6.9}
\begin{split}
&\left[\prod^{l}_{j=0}((k+j+1/2)^{2}-(n-1)^{2}/4-\Delta_{\mathbb{H}})\right]^{-1}\\
=&\frac{\pi^{-m}\left(\sinh \rho\right)^{2-2m}}{(2k+1+l)l4\Gamma(l)}
\left[\frac{\Gamma(m+k)\Gamma(k+1)}{\Gamma(2k+l+1)}\left(\sinh\frac{\rho}{2}\right)^{2m-2k-4}+\right.\\
&\left.(2k+1+l)\sum^{m-2-k-l}_{j=1}\frac{j\Gamma(m+k)\Gamma(m-1-k-l)\Gamma(m-j-l)}{\Gamma(j+1)\Gamma(m-k-l-j)\Gamma(m+k+1-j)}
\left(\sinh\frac{\rho}{2}\right)^{2j+2l-2}\right].\\
=&\frac{\pi^{-m}\left(\sinh \rho\right)^{2-2m}}{4\Gamma(l+1)}
\sum^{m-2-k-l}_{j=0}\frac{\Gamma(m+k)\Gamma(m-1-k-l)\Gamma(m-1-j-l)}{\Gamma(j+1)\Gamma(m-1-k-l-j)\Gamma(m+k-j)}
\left(\sinh\frac{\rho}{2}\right)^{2j+2l}
\end{split}
\end{equation}
The desired result follows.
\end{proof}

\begin{lemma} \label{lm6.3}
Let $n=2m$ and $1<k<n/2$. The Green's function of $P_{k}$ satisfies
\begin{equation*}
\begin{split}
P^{-1}_{k}(\rho)=&\frac{4^{k-1}}{\gamma_{2m}(2k)(\sinh\rho)^{2m-2}}\sum^{m-1-k}_{j=0}\frac{\Gamma(m)}{\Gamma(j+1)\Gamma(m-j)}
\left(\sinh\frac{\rho}{2}\right)^{2j+2k-2}\\
\leq&\frac{1}{\gamma_{n}(2k)}\left[\left(\frac{1}{2\sinh\frac{\rho}{2}}\right)^{n-2k}-
\left(\frac{1}{2\cosh\frac{\rho}{2}}\right)^{n-2k}\right].
\end{split}
\end{equation*}
\end{lemma}
\begin{proof}
Choosing $k=0$ and $l=k$, we have
\begin{equation}\label{6.10}
\begin{split}
P^{-1}_{k}=&\left[\prod^{k-1}_{j=0}((j+1/2)^{2}-(n-1)^{2}/4-\Delta_{\mathbb{H}})\right]^{-1}\\
=&\frac{1}{4\Gamma(k)\pi^{m}\left(\sinh \rho\right)^{2m-2}}
\sum^{m-1-k}_{j=0}\frac{\Gamma(m)\Gamma(m-k)}{\Gamma(j+1)\Gamma(m-j)}
\left(\sinh\frac{\rho}{2}\right)^{2j+2k-2}\\
=&\frac{4^{k-1}}{\gamma_{2m}(2k)(\sinh\rho)^{2m-2}}\sum^{m-1-k}_{j=0}\frac{\Gamma(m)}{\Gamma(j+1)\Gamma(m-j)}
\left(\sinh\frac{\rho}{2}\right)^{2j+2k-2}.
\end{split}
\end{equation}
On the other hand,
\begin{equation}\label{6.11}
\begin{split}
&\frac{1}{\gamma_{n}(2k)}\left[\left(\frac{1}{2\sinh\frac{\rho}{2}}\right)^{n-2k}-
\left(\frac{1}{2\cosh\frac{\rho}{2}}\right)^{n-2k}\right]\\
=&\frac{4^{k-1}}{\gamma_{2m}(2k)(\sinh\rho)^{2m-2}}\left[\left(\cosh\frac{\rho}{2}\right)^{2m-2}\left(\sinh\frac{\rho}{2}\right)^{2k-2}
-\left(\cosh\frac{\rho}{2}\right)^{2k-2}\left(\sinh\frac{\rho}{2}\right)^{2m-2}\right]\\
=&\frac{4^{k-1}}{\gamma_{2m}(2k)(\sinh\rho)^{2m-2}}\left[\left(1+\sinh^{2}\frac{\rho}{2}\right)^{m-1}\left(\sinh\frac{\rho}{2}\right)^{2k-2}
-\right.\\
&\left.\left(1+\sinh^{2}\frac{\rho}{2}\right)^{k-1}\left(\sinh\frac{\rho}{2}\right)^{2m-2}\right]\\
=&\frac{4^{k-1}}{\gamma_{2m}(2k)(\sinh\rho)^{2m-2}}\left[\sum^{m-1}_{j=0}\frac{\Gamma(m)}{\Gamma(j+1)\Gamma(m-j)}\left(\sinh\frac{\rho}{2}\right)^{2j+2k-2}
-\right.\\
&\left.\sum^{k-1}_{j=0}\frac{\Gamma(k)}{\Gamma(j+1)\Gamma(k-j)}\left(\sinh\frac{\rho}{2}\right)^{2j+2m-2}\right].\\
\end{split}
\end{equation}
Therefore,
\begin{equation}\label{6.12}
\begin{split}
&\frac{1}{\gamma_{n}(2k)}\left[\left(\frac{1}{2\sinh\frac{\rho}{2}}\right)^{n-2k}-
\left(\frac{1}{2\cosh\frac{\rho}{2}}\right)^{n-2k}\right]-P^{-1}_{k}\\
=&\frac{4^{k-1}}{\gamma_{2m}(2k)(\sinh\rho)^{2m-2}}\left[\sum^{m-1}_{j=m-k}\frac{\Gamma(m)}{\Gamma(j+1)\Gamma(m-j)}\left(\sinh\frac{\rho}{2}\right)^{2j+2k-2}
-\right.\\
&\left.\sum^{k-1}_{j=0}\frac{\Gamma(k)}{\Gamma(j+1)\Gamma(k-j)}\left(\sinh\frac{\rho}{2}\right)^{2j+2m-2}\right].\\
=&\frac{4^{k-1}}{\gamma_{2m}(2k)(\sinh\rho)^{2m-2}}\left[\sum^{k-1}_{j=0}\frac{\Gamma(m)}{\Gamma(m-k+j+1)\Gamma(k-j)}\left(\sinh\frac{\rho}{2}\right)^{2j+2m-2}
-\right.\\
&\left.\sum^{k-1}_{j=0}\frac{\Gamma(k)}{\Gamma(j+1)\Gamma(k-j)}\left(\sinh\frac{\rho}{2}\right)^{2j+2m-2}\right]\\
=&\frac{4^{k-1}}{\gamma_{2m}(2k)(\sinh\rho)^{2m-2}}\sum^{k-1}_{j=0}\frac{1}{\Gamma(k-j)}
\left(\frac{\Gamma(m)}{\Gamma(m-k+j+1)}-\frac{\Gamma(k)}{\Gamma(j+1)}\right)\left(\sinh\frac{\rho}{2}\right)^{2j+2m-2}\\
\geq&0.
\end{split}
\end{equation}
This completes the proof of Lemma \ref{lm6.3}.
\end{proof}

\subsection{Case 2: $n$ is odd }
The proof depends on the following facts: if we denote by $h_{n}(t,\rho)$ the heat kernel of $n$-dimension hyperbolic space, then (see \cite{d2}, page 178)
\begin{equation}\label{6.13}
h_{n}(t,\rho)=\sqrt{2}e^{\frac{2n-1}{4}t}\int^{\infty}_{\rho}\frac{\sinh r}{\sqrt{\cosh r-\cosh\rho}}h_{n+1}(t,r)dr.
\end{equation}
Therefore, if we denote by $G_{n}(\rho,\lambda) (\lambda\geq-\frac{(n-1)^{2}}{4})$ be the Green's function of $\lambda-\Delta_{\mathbb{H}}$ of dimension $n$, then by (\ref{6.13}),
 \begin{equation}\label{6.14}
\begin{split}
G_{n}(\rho,\lambda)=&\int^{\infty}_{0}e^{-\lambda t}h_{n}(t,\rho)dt\\
=&\sqrt{2}
\int^{\infty}_{\rho}\frac{\sinh r}{\sqrt{\cosh r-\cosh\rho}}\left(\int^{\infty}_{0}e^{-\lambda t+\frac{2n-1}{4}t}h_{n+1}(t,r)dt\right)dr\\
=&\sqrt{2}
\int^{\infty}_{\rho}\frac{\sinh r}{\sqrt{\cosh r-\cosh\rho}}G_{n+1}(r,\lambda-(2n-1)/4)dr
\end{split}
\end{equation}
and thus
\begin{equation}\label{6.15}
\begin{split}
&G_{n}(\rho,(k+i+1/2)^{2}-(n-1)^{2}/4)\\
=&\sqrt{2}
\int^{\infty}_{\rho}\frac{\sinh r}{\sqrt{\cosh r-\cosh\rho}}G_{n+1}(r,(k+i+1/2)^{2}-n^{2}/4)dr.
\end{split}
\end{equation}

By (\ref{6.15}) and Lemma \ref{lm6.2}, we have that if $n=2m-1$, then
\begin{equation}\label{6.16}
\begin{split}
&\left[\prod^{l-1}_{j=0}((k+j+1/2)^{2}-(n-1)^{2}/4-\Delta_{\mathbb{H}})\right]^{-1}\\
=&\sqrt{2}
\int^{\infty}_{\rho}\left[\frac{1}{4\Gamma(l)\pi^{m}\left(\sinh r\right)^{2m-2}}
\sum^{m-1-k-l}_{j=0}\frac{\Gamma(m+k)\Gamma(m-k-l)\Gamma(m-j-l)}{\Gamma(j+1)\Gamma(m-k-l-j)\Gamma(m+k-j)}\cdot\right.\\
&\left.\left(\sinh\frac{r}{2}\right)^{2j+2l-2}\right]\frac{\sinh r}{\sqrt{\cosh r-\cosh\rho}}dr,
\end{split}
\end{equation}
where $l\in\{1,2,\cdots,m-1\}$ and $k\in\{0,1,\cdots, m-1-l\}$.

\medskip

\begin{lemma} \label{lm6.4}
Let $n=2m-1\geq3$ and $1<k<n/2$. The Green's function of $P_{k}$ satisfies
\begin{equation*}
\begin{split}
P^{-1}_{k}(\rho)=&\frac{4^{k-1}\sqrt{2}}{\gamma_{2m}(2k)}\sum^{m-1-k}_{j=0}\frac{\Gamma(m)}{\Gamma(j+1)\Gamma(m-j)}
\int^{\infty}_{\rho}
\left(\sinh\frac{r}{2}\right)^{2j+2k-2}\frac{(\sinh r)^{3-2m}}{\sqrt{\cosh r-\cosh\rho}}dr\\
\leq&\frac{1}{\gamma_{n}(2k)}\left[\left(\frac{1}{2\sinh\frac{\rho}{2}}\right)^{n-2k}-
\left(\frac{1}{2\cosh\frac{\rho}{2}}\right)^{n-2k}\right].
\end{split}
\end{equation*}
\end{lemma}
\begin{proof}
With the same argument as in the proof of Lemma \ref{lm6.3}, we have, by (\ref{6.16}),
\begin{equation}\label{6.11}
\begin{split}
P^{-1}_{k}(\rho)=&\frac{4^{k-1}\sqrt{2}}{\gamma_{2m}(2k)}\sum^{m-1-k}_{j=0}\frac{\Gamma(m)}{\Gamma(j+1)\Gamma(m-j)}
\int^{\infty}_{\rho}
\left(\sinh\frac{r}{2}\right)^{2j+2k-2}\frac{(\sinh r)^{3-2m}}{\sqrt{\cosh r-\cosh\rho}}dr\\
\leq&\frac{\sqrt{2}}{\gamma_{2m}(2k)}
\int^{\infty}_{\rho}\left[\left(\frac{1}{2\sinh\frac{\rho}{2}}\right)^{2m-2k}-
\left(\frac{1}{2\cosh\frac{\rho}{2}}\right)^{2m-2k}\right]\frac{\sinh r}{\sqrt{\cosh r-\cosh\rho}}dr.
\end{split}
\end{equation}
Substituting $t=\sqrt{(\cosh r-\cosh\rho)/2}$ yields
\begin{equation*}
\begin{split}
P^{-1}_{k}(\rho)\leq&\frac{1}{\gamma_{2m}(2k)2^{2m-2k}}\int^{\infty}_{0}\left[\left(\frac{1}{t^{2}+\sinh^{2}\frac{\rho}{2}}\right)^{m-k}-
\left(\frac{1}{t^{2}+\cosh^{2}\frac{\rho}{2}}\right)^{m-k}\right]4dt\\
=&\frac{4}{\gamma_{2m}(2k)2^{2m-2k}}\cdot\frac{(2m-2k-3)!!\pi}{2\cdot(2m-2k-2)!!}
\left[\left(\frac{1}{\sinh\frac{\rho}{2}}\right)^{2m-2k-1}-
\left(\frac{1}{\cosh\frac{\rho}{2}}\right)^{2m-2k-1}\right]\\
=&\frac{1}{\gamma_{2m-1}(2k)}\left[\left(\frac{1}{2\sinh\frac{\rho}{2}}\right)^{n-2k}-
\left(\frac{1}{2\cosh\frac{\rho}{2}}\right)^{n-2k}\right].
\end{split}
\end{equation*}
To get the first equality, we use the identity (see e.g. \cite{gr}, page 324)
\[
\int^{\infty}_{0}\frac{dt}{(t^{2}+a^{2})^{n}}=\frac{(2n-3)!!}{2\cdot(2n-2)!!}\frac{\pi}{a^{2n-1}},\;\;a>0.
\]
The proof of Lemma \ref{lm6.4} is thereby completed.
\end{proof}

\textbf{Proof of Theorem \ref{th1.3}} Combining  Lemma \ref{lm6.3} and \ref{lm6.4}, Theorem \ref{th1.3} follows.

\end{document}